\documentclass[12pt, twoside,reqno,a4paper]{amsart}
\usepackage{amsthm}
\theoremstyle{plain}

\newtheorem*{theorem*}{Theorem}
\newtheorem{theorem}{Theorem}
\def\textmatrix#1&#2\\#3&#4\\{\bigl({#1 \atop #3}\ {#2 \atop #4}\bigr)}
\newcommand{\bydef}{\stackrel{\rm def}{=}}

\newcommand{\cle}{{\mathcal{E}}}
\newcommand{\clh}{\mathcal{H}}
\newcommand{\clk}{{\mathcal{K}}}

\newcommand{\clm}{{\mathcal{M}}}

\newcommand{\zbar}{{\overline{z}}}
\newcommand{\wbar}{{\overline{w}}}

\def\C{\mathbb{C}}

\def\v{\varphi}
\def\bl{\boldsymbol}

\def\D{\mathbb D}

\def\m{\mathcal}
\def\mb{\mathbb}

\def\i{\prime}

\def\c{curvature~}

\newtheorem{thm}{Theorem}[section]
\newtheorem{cor}[thm]{Corollary}
\newtheorem{lem}[thm]{Lemma}
\newtheorem{prop}[thm]{Proposition}
\newtheorem{defn}[thm]{Definition}
\newtheorem{rem}[thm]{Remark}

\newcommand{\be}{\begin{equation}}
\newcommand{\ee}{\end{equation}}
\newcommand{\bea}{\begin{eqnarray}}
\newcommand{\eea}{\end{eqnarray}}
\newcommand{\Bea}{\begin{eqnarray*}}
\newcommand{\Eea}{\end{eqnarray*}}

\newcounter{cnt1}
\newcounter{cnt2}
\newcounter{cnt3}
\newcommand{\blr}{\begin{list}{$($\roman{cnt1}$)$}
 {\usecounter{cnt1} \setlength{\topsep}{0pt}
 \setlength{\itemsep}{0pt}}}
\newcommand{\bla}{\begin{list}{$($\alph{cnt2}$)$}
 {\usecounter{cnt2} \setlength{\topsep}{0pt}
 \setlength{\itemsep}{0pt}}}
\newcommand{\bln}{\begin{list}{$($\arabic{cnt3}$)$}
 {\usecounter{cnt3} \setlength{\topsep}{0pt}
 \setlength{\itemsep}{0pt}}}
\newcommand{\el}{\end{list}}

\font\myfont=cmr12 at 16pt

\usepackage{enumerate}
\usepackage{color}

\title[The Agler-Young class]{\myfont The Agler-Young class}

\author[Bhattacharyya]{Tirthankar Bhattacharyya}
\address[Bhattacharyya]{Department of Mathematics, Indian Institute of Science, Bangalore 560 012, India.}
\email{tirtha@iisc.ac.in}

\author[Shyam Roy]{Subrata Shyam Roy}
\address[Shyam Roy]{Indian Institute of Science Education and Research Kolkata, Campus Road, Mohanpur, West Bengal 741 246, India. }
\email{ssroy@iiserkol.ac.in}

\author[Yadav]{Tapesh Yadav}
\address[Yadav]{Department of Mathematics, Indian Institute of Science, Banaglore 560 012, India.}
\email{tapeshyadav1@gmail.com}	

\begin{document}

\thanks{This research is supported by University Grants Commission, India via CAS}

\begin{abstract}

\vspace*{5mm}

 This note introduces a special class of tuples of bounded operators on a Hilbert space. It is called the Agler Young class. Major results about this class include a Wold decomposition and a dilation theorem. The structure of the dilation is completely spelt out. A characterization of this class using the hereditary functional calculus of Agler is obtained and examples are discussed. Toeplitz operators play a major role in this note. An Agler-Young pair arising from a truncated Toeplitz operator is characterized. Thus, we extend results obtained in the case of commuting operators by several authors over many decades to the non-commutative situation. The results for the commuting case can be recovered as special cases.

\end{abstract}

\maketitle

\section{Introduction}
\subsection{Block Toeplitz operators}
In Hilbert space operator theory, it is important to identify a special class of nice operators or operator tuples and to decode its structure in terms of simpler objects in the same class. An example of this kind of endeavours which has stood the test of time is the Wold decomposition of an isometry: given an isometry $A$ on a Hilbert space $\mathcal H$, there are two $A$-reducing subspaces $\mathcal H_1$ and $\mathcal H_2$ such that $\mathcal H = \mathcal H_1 \oplus \mathcal H_2$,  $A |_{\mathcal H_1} $ is a unitary and $A |_{\mathcal H_2} $ is a unilateral shift (of some multiplicity, possibly infinite). This simple theorem has profound applications in different areas of mathematics and statistics, see \cite{BDF}, \cite{BKS}, \cite{CPS}, \cite{GG}, \cite{GS}, \cite{HL}, \cite{KMN}, \cite{KM} and \cite{KO}.

It is clear, then, that for a Wold type decomposition to work for an operator (or a tuple of operators), a crucial ingredient is the unilateral shift (of multiplicity one or higher). This is the simplest example of a Toeplitz operator. Let us collect some preliminaries of {\em block Toeplitz operators} from the seminal paper of Rabindranathan \cite{Dadu}. Let $\mathcal E$ be a  Hilbert space. Let $\mathcal O(\mathbb D, \mathcal E)$ be the class of all $\mathcal E$ valued holomorphic functions on the open unit disc $\mathbb D = \{ z \in \mathbb C : |z| < 1\}$. Let
$$H^2(\mathcal E) = \{ f = \sum_{n=0}^\infty a_k z^k \in \mathcal O(\mathbb D, \mathcal E) : \{a_k\}_{k\ge 0} \subset \mathcal E \mbox{ and } \| f \|^2 = \sum_{n=0}^\infty \| a_k \|^2 < \infty\}.$$
If $\mathcal E = \mathbb C$, we just write $H^2$. The space $H^2(\mathcal E)$ is isometrically isomorphic to $H^2 \otimes \mathcal E$ and sometimes this identification will be used without further mentioning it.

\begin{defn} \label{BlockT} If $\v$ is an $L^\infty(\mathcal B(\mathcal E))$ valued function defined on the unit circle $\mathbb T = \{z\in \mathbb C: |z| = 1\}$, then the Toeplitz operator $T_\v$ on $H^2(\mathcal E)$ is defined as
$$ T_\v g = P_+ M_\v g \mbox{ for } g \in H^2(\mathcal E)$$
where $M_\v$ is the multiplication operator on $L^2(\mathcal E)$ and $P_+$ is the projection from $L^2(\mathcal E)$ onto $H^2(\mathcal E)$. \end{defn}

Here the $L^\infty$ and $L^2$ are with respect to the normalized Haar measure of the circle group and we consider the natural embedding of $H^2(\mathcal E)$ as a subspace of $L^2(\mathcal E)$. If dim $\mathcal E > 1$, then the Toeplitz operator $T_\varphi$, in the definition above, is popularly called the $block$ Toeplitz operator. We shall not always do that. It will always be clear from the context whether the Hardy space concerned consists of scalar valued functions or vector valued functions.

\subsection{A canonical element}
In search of the right class of operator tuples that will play the role that the unilateral shift so efficiently played in the Wold decomposition theorem, consider the following example.

Every $L^2(\mathcal B(\mathcal E))$ function has a Fourier series expansion  with respect to the natural basis $\{ e_n(\theta) = \exp({in\theta}), n \in \mathbb Z\}$ of $L^2$ (see, for example, Lemma 4.2.8 of \cite{SG}). An $L^\infty(\mathcal B(\mathcal E))$ function $\varphi$ is in $L^2(\mathcal B(\mathcal E))$. Let
 $$\varphi = \sum_{n \in \mathbb Z} A_n e_n$$
 be its Fourier series expansion. Here, the $A_n$ are from $\mathcal B(\mathcal E)$. The function $\varphi$ is said to be holomorphic if all the negative Fourier coeffcients are $0$. Let $H^\infty\big(\m B (\cle)\big)$ denote the subalgebra of $L^\infty(\mathcal B(\mathcal E))$ that consists of all such holomorphic elements of $L^\infty(\mathcal B(\mathcal E))$. Consider an $(n-1)$-tuple $\bl f = (f_1, f_2, \ldots , f_{n-1})$ of functions from $H^\infty\big(\m B (\cle)\big)$. Set $$\v_i(z) = z{f_i}(z) + {f_{n-i}}(z)^*  \mbox{ for } i=1,2, \ldots ,n-1.$$ This tuple of functions $(\v_1, \v_2, \ldots ,\v_{n-1})$ defined on  $\mathbb T$ and taking values in $L^\infty(\mathcal B(\mathcal E))$ is called the {\em co-analytic extension} of $(f_1, f_2, \ldots ,f_{n-1})$.
Let $S_i = T_{\v_i}$ for $i=1,2, \ldots , n-1$ and let $S_n$ be the pure isometry $M_z$ on $H^2(\cle)$. For $z$ on the unit circle,
$$ \v_{n-i}(z)^* z = (\zbar f_{n-i}(z)^* + f_i(z)) z = f_{n-i}(z)^* + zf_i(z) = \v_i(z).$$
So, $\underline{S} = (S_1, S_2, \ldots,S_n)$ has the property that $S_n$ is a unilateral shift (of multiplicity equal to the dimension of the space $\cle = \mathcal D_{S_n^*}$ and the rest of the $S_i$ are Toeplitz operators that satisfy
$$ S_i = S_{n-i}^*S_n.$$

\begin{defn} The tuple $\underline{S}$ described above is called the canonical Agler-Young isometry associated with the function tuple $\bl f$. The functions $f_1, f_2, \ldots , f_{n-1}$ are called the fundamental functions of the canonical Agler-young isometry $\underline{S}$. \end{defn}

\subsection{The Agler-Young class} We shall work towards a Wold decomposition of an Agler-Young isometry defined below. If $T$ is a contraction, let $D_{T} = (I - T^*T)^{1/2}$ and $\mathcal D_{T} =  \overline{\rm Ran} D_{T}$. This notation goes back to Sz.-Nagy.

\begin{defn} For a tuple of bounded operators $\underline{S} = (S_1, S_2, \ldots,S_n)$ on a Hilbert space $\mathcal H$ such that $S_n$ is a contraction, the $n-1$ operator equations
\begin{equation} \label{AYdef}
S_i - S_{n-i}^*S_n = D_{S_n} X_i D_{S_n} \text{ for } i=1,2, \ldots , n-1 \text{ and } X_{i} \in \mathcal B(\mathcal D_{S_n})\end{equation}
are called its fundamental equations. The operator tuple $\underline{S}$ is said to be in the Agler-Young class $AY_n$ or called an Agler-Young tuple if $S_n$ is a contraction and $\underline{S}$ satisfies the fundamental equations.  \end{defn}

For an Agler-Young tuple $(S_1, S_2, \ldots,S_n)$, the solution tuple $(X_1, X_2, \ldots ,X_{n-1})$ is unique and will be called the {\em fundamental operator tuple} of $\underline{S}$. An Agler-Young tuple is called an {\em Agler-Young isometry} if $S_n$ is an isometry. An Agler-Young isometry is called a $pure$ Agler-Young isometry if the isometry $S_n$ is pure, i.e., $(S_n^*)^k$ converges strongly to $0$ as $k \rightarrow \infty$, i.e., $S_n$ is a shift of some multiplicity. The canonical Agler-Young isometry defined above is a pure Agler-Young isometry.

One still needs the concept similar to that of a unitary, i.e., the other ingredient in Wold decomposition.

\begin{defn}
An Agler-Young isometry is called an Agler-Young unitary if $S_n$ is a unitary operator. \end{defn}

An Agler-Young unitary $\underline{S} = (S_1, S_2, \ldots,S_n)$ has the property that $S_n$ is a unitary operator that commutes with each $S_i$ for $i=1,2, \ldots, n-1$. Indeed,
$$S_nS_i = S_n S_{n-i}^*S_n = S_n(S_i^*S_n)^*S_n = S_nS_n^*S_iS_n = S_iS_n.$$

From definition, it is clear that a tuple of bounded operators $\underline{S} = (S_1, S_2, \ldots,S_n)$ is an Agler-Young isometry if and only if $S_n$ is an isometry and $S_i = S_{n-i}^*S_n$ for $i=1,2, \ldots ,n-1$. This is a characterization which can be immediately used to prove that the  restriction of an Agler-Young isometry to an invariant subspace is again an Agler-Young isometry. Indeed, consider an Agler-Young isometry $\underline{S} = (S_1, S_2, \ldots,S_n)$ on a Hilbert space $\mathcal H$ and let $\clm$ be an invariant subspace. Let $S_i|_\clm = T_i$ for $i=1,2, \ldots ,n$. Obviously then with respect to the decomposition $\clh = \clm \oplus \clm^\perp$, the $S_i$ have the decomposition: $S_i = \textmatrix T_i & A_i \\ 0 & B_i \\$ for suitable operators $A_i$ and $B_i$. Clearly, $T_n$ is an isometry. We use $S_i = S_{n-i}^*S_n$ to get
$$\left(
    \begin{array}{cc}
      T_i & A_i \\
      0 & B_i \\
    \end{array}
  \right) = \left(
    \begin{array}{cc}
      T_{n-i}^* & 0 \\
      A_{n-i}^* & B_{n-i}^* \\
    \end{array}
  \right) \left(
    \begin{array}{cc}
      T_n & A_n \\
      0 & B_n \\
    \end{array}
  \right) = \left(
    \begin{array}{cc}
      T_{n-i}^*T_n & \star \\
      \star & \star \\
    \end{array}
  \right)$$
  and hence the tuple $\underline{T} = (T_1, T_2, \ldots, T_n)$ satisfies the characterization mentioned above. Needless to say that a canonical Agler-Young isometry defined earlier is an Agler-Young isometry.

\subsection{The Wold decomposition}

The first main theorem of this work is a structure theorem for Agler-Young isometries in the style of Wold. It is stated below and proved in Section 2.

\begin{theorem}[\textbf{The Wold decomposition theorem} for an Agler-Young isometry] \label{Wold}
Let $\underline{S} = (S_1, S_2. \ldots,S_n)$ be an Agler-Young isometry on $\mathcal H$ with $\dim\mathcal D_{S_n^*} < \infty$.
Then there is a unique orthogonal decomposition $\mathcal H = \mathcal H_1 \oplus \mathcal H_2$ of the Hilbert space $\mathcal H$ such that
\begin{enumerate}
\item [(a)]$\mathcal H_1$ and $\mathcal H_2$ are common reducing subspaces for the $S_i$,
 \item [(b)]$(S_1|_{\m H_1}, S_2|_{\m H_1}, \ldots ,S_n|_{\m H_1})$ is an Agler-Young unitary.
 \item [(c)] $S_n|_{\mathcal H_2}$ is a pure isometry (a unilateral shift) $V$. There is a unitary operator $W:\m H_2\to H^2(\m D_{V^*})$ and a unique $(n-1)$-tuple $\bl f=(f_1, f_2, \ldots , f_{n-1})$ of $\mathcal B(\mathcal D_{V^*})$ valued bounded holomorphic functions such that the tuple
$$ (WS_1|_{\m H_2}W^*, WS_2|_{\m H_2}W^*, \ldots ,WS_n|_{\m H_2}W^*)$$
is the canonical pure Agler-Young isometry associated with $\bl f$. Further, the following relation is satisfied for every $i =1, 2, \ldots , n-1:$
\begin{equation} \label{determinant} S_{n-i}^* - S_i S_n^* = 0_{\mathcal H_1} \oplus W^*\left( T_{f_i} (I - T_zT_z^*) \right)W.\end{equation}

\end{enumerate}

\end{theorem}

The finite dimensionality condition is redundant in case $\underline{S}$ is a commuting Agler-Young isometry, see Corollary \ref{comAYiso}.

\subsection{Dilation} One of the most well-known results in single variable operator theory is the Sz.-Nagy dilation of a contraction to an isometry. In the above, we have defined an Agler-Young tuple.   Our next major result is about dilating such a tuple to an Agler-Young isometry.

\begin{defn} Let $\mathcal{H} \subset \mathcal{K}$ be two Hilbert spaces.
Suppose $\underline{S} = (S_1, S_2, \ldots,S_n)$ and $\underline{V} = (V_1, V_2, \ldots
, V_n)$ are tuples of bounded operators acting on $\mathcal{H}$
and $\mathcal{K}$ respectively, that is, $S_i \in
\mathcal{B}(\mathcal{H})$ and $V_i \in \mathcal{B}(\mathcal{K})$.
The operator tuple $\underline{V}$ is called a $dilation$ of the operator tuple
$\underline{S}$ if
\begin{equation} S_{i_1} S_{i_2} \ldots S_{i_k} h = P_{\mathcal{H}} V_{i_1} V_{i_2} \ldots
V_{i_k} h \mbox{ for } h \in \mathcal{H}, \;\; k\ge 1 \mbox{
and } 1 \le i_1,i_2, \ldots , i_k \le n. \label{dilation}
\end{equation}
If, moreover, $\clh$ is an invariant subspace for each $V_i^*$, then $\clh$ is called a co-invariant subspace of $\underline{V}$. The dilation is called minimal if
$$\clk = \overline{\mathrm{span}} \{V_{i_1} V_{i_2} \ldots V_{i_k} h \mbox{ for } h \in \mathcal{H}, \;\; k\ge 1 \mbox{ and } 1 \le i_1,i_2, \ldots , i_k \le n\}.$$

\end{defn}

The definition of minimality is natural because the dilation space $\clk$ has to contain all elements of the form $V_{i_1} V_{i_2} \ldots V_{i_k} h$ and hence could not be any smaller than what is described in the definition.

\begin{rem} Note that if the $V_i^*$ do leave $\clh$ as an invariant subspace, then equation \eqref{dilation} is automatically satisfied. Indeed, in this case, the $V_i$ have the decomposition $ V_i =  \textmatrix S_i & 0 \\ \star & \star \\$ with respect to the decomposition $\clk = \clh \oplus (\clk \ominus \clh)$ of the space $\clk$ which immediately implies that
 $$ V_{i_1} V_{i_2} \ldots V_{i_k} = \begin{pmatrix} S_{i_1} S_{i_2} \ldots S_{i_k} & 0 \\ \star & \star \end{pmatrix}$$
 which in turn implies \eqref{dilation}. \end{rem}
 Dilation of an operator is a highly successful tool and was inrtoduced by Sz.-Nagy in \cite{NagyDilation} where he proved that a contraction can be dilated to an isometry. Constructing explicit dilation is always a challenge and has been done in only a few cases.
  \begin{enumerate}
  \item The isometric dilation (named after Sz.-Nagy because he proved its existence) for a contraction was constructed by Sch$\ddot{\mbox{a}}$ffer in \cite{Schaffer}.
  \item The commuting isometric dilation for a pair of commuting contractions was constructed by And$\hat{\mbox{o}}$ in \cite{Ando}.
      \item The dilation of a contractive tuple to a tuple of isometries with orthogonal ranges was constructed by Popescu in \cite{Popescu}.
       \item The $\Gamma$-isometric dilation for a $\Gamma$-contraction was constructed by Bhattacharyya, Pal and Shyam Roy in \cite{BhPSR}, although the existence had been shown by Agler and Young earlier in \cite{AYJoT}.  \end{enumerate}

           We construct here an explicit dilation for an Agler-Young tuple. The dilation is an Agler-Young isometry. The theorem is stated below and proved in Section 4.

\begin{theorem}[\textbf{The dilation theorem} for an Agler-Young contraction] \label {schaeffer}
Let $$\underline{S} = (S_1, S_2, \ldots,S_n)$$
be an Agler-Young contraction on a Hilbert space
$\mathcal{H}$. Let $\bl X = (X_1, X_2, \ldots , X_{n-1})$ be the $(n-1)$-tuple of operators obtained from defining equation \eqref{AYdef} with $X_i\in\m B (\mathcal D_{S_n}),i=1,\ldots n-1.$
Let
$$\m K_0=\m H\oplus\m D_{S_n}\oplus\m D_{S_n}\oplus\ldots=\m H\oplus\ell^2(\m D_{S_n}).$$
Consider the operator tuple $\underline{V}^{\bl X} = (V_1^{\bl X}, V_2^{\bl X}, \ldots ,V_{n-1}^{\bl X}, V_n)$ defined on
$\mathcal{K}_0$ by
$$ V_i^{\bl X}(h_0,h_1,h_2,\ldots)= (S_ih_0,X_{n-i}^*{D}_{S_n}h_0+X_ih_1, X_{n-i}^*h_1+X_ih_2, X_{n-i}^*h_2+X_ih_3,\ldots )$$
 for $i=1,2, \ldots ,n-1$ and
$$ V_n(h_0,h_1,h_2,\ldots)  =  (S_nh_0,{D}_{S_n}h_0,h_1, h_2,\ldots).$$

Consider $\clh$ as a subspace of $\clk_0$ by identifying $h$ of $\clh$ with the vector $h \oplus 0\oplus 0\oplus\ldots$ of $\clk_0,$ where $0\oplus 0\oplus\ldots$ is the identically zero sequence in $\ell^2(\mathcal{D}_{S_n}).$
Then \begin{enumerate}
\item[(1)] $\clh$ is a co-invariant subspace of $\underline{V}^{\bl X}$ and $\underline{V}^{\bl X}$ is an Agler-Young isometric dilation of $\underline{S}$.
\item[(2)] If $(W_1, W_2, \ldots ,W_{n-1}, V_n)$ is any Agler-Young isometric dilation for $\underline{S}$ on $\clk_0$ whose action is such that $\clh$ is a co-invariant subspace, then  $W_i = V_i^{\bl X}$ for $i=1,2, \ldots ,n-1.$
\item[(3)] If $ (W_1, W_2, \ldots ,W_n)$ is an Agler-Young isometric dilation of $\underline{S}=(S_1,S_2\ldots,S_n),$ where $W_n$ is a minimal isometric dilation of $S_n,$ then $(W_1, W_2, \ldots ,W_n)$ is unitarily equivalent to $(V_1^{\bl X}, V_2^{\bl X}, \ldots ,V_{n-1}^{\bl X}, V_n).$
\end{enumerate}
\end{theorem}

We would like to emphasize two important points of the theorem above.
 \begin{enumerate}
\item[(a)] By part (1), the dilation takes place on the minimal isometric dilation space of the contraction $S_n$ and it is automatically minimal because the dilation space could not be any smaller.
\item[(b)] Parts (2) and (3) give a natural uniqueness. \end{enumerate}

\subsection{Organization}   A satisfactory characterization of the Agler-Young class is obtained using hereditary polynomials introduced by Agler in his landmark paper \cite{AglerFamily}, where he  outlined an abstract approach to model theory. This characterization enables us to conclude that a $\Gamma_n$-contraction (see \cite{SS}) is in the class $AY_n$ and a tetrablock contraction (see \cite{BhTetra}) is a member of $AY_3.$  Therefore, the Agler-Young class provides a broader stage to study these classes of operators which have received considerable attention recently.

Since it is impossible to describe all important results without going into all the details, further definitions and results are introduced in appropriate places in the paper.

Section 2 proves the Wold decomposition of an Agler-Young isometry. Section 3 extends an Agler-Young isometry to an Agler-Young unitary and finds a complete set of invariants for a pure Agler-Young isometry. Section 4 proves that any Agler-Young contraction can be dilated to an Agler-Young isometry, the dilation is unique in a natural way and the dilation has a nice explicit structure as mentioned above.

Section 5 gives an alternative description of the dilation for a pure Agler-Young tuple and a functional model. This section has an invariant subspace theorem following the classical work of Beurling, Lax and Halmos.

Section 6 proves a von Neumann type inequality with respect to the hereditary functional calculus, relates the Agler-Young class to a certain family and shows that the Agler-Young isometries are the extremals of this family.

Section 7 introduces the connection with truncated Toeplitz operators which, after being introduced by Sarason in \cite{Sarason}, has matured into a major theme of research. We characterize those Agler-Young pairs, the first component of which is a truncated Toeplitz operator.

Section 8 deals with the commutative case which is what has been studied so far in the literature and some of the existing results are obtained as special cases of the non-commutative theory developed in this article.

\section{Proof of the Wold decomposition}

The proof of the Wold decomposition theorem will involve several lemmas. We recall that for an Agler-Young isometry $\underline{S} = (S_1, S_2, \ldots,S_n)$, the condition \eqref{AYdef} is the same as
\begin{equation} \label{ar}
S_n^*S_i=S_{n-i}^* \mbox{ for } i=1,\ldots,n-1.
\end{equation}

\begin{lem}\label{conj}
Suppose that $\{A_k\}$ and $\{B_k\}$ are sequences of bounded operators on a Hilbert space $\m H$ which converge to $A$ and $B$ respectively, in the strong operator topology of $\m B(\m H)$ and $F$ is a finite rank operator on $\m H.$ Then, the sequence $\{A_kFB_k^*\}$ converges to $AFB^*$ in the norm topology of $\m B(\m H).$
\end{lem}
\begin{proof}
It is enough to prove the result for a rank one operator. For $x, y\in \m H,$ consider the rank one operator on $\m H$ defined by $(x\otimes y)h=\langle h,y\rangle x$ for $h\in\m H.$ If $\{x_k\}$ and $\{y_k\}$ are sequences of vectors in $\m H$ converging to $x$ and $y$ in the norm of $\m H,$ respectively, then it is easy to see that $\{x_k\otimes y_k\}$ converges to $x\otimes y$ in the norm topology of $\m B(\m H). $ Consequently, by hypothesis, $\{A_kx\otimes B_ky\}$ converges to $Ax\otimes By$ in the norm topology of $\m B(\m H)$ for any $x,y\in\m H.$ Since $A(x\otimes y)B^*=Ax\otimes By$, we conclude that $\{A_k(x\otimes y)B_k^*\}$ converges to $A(x\otimes y)B^*$ in the norm topology of $\m B(\m H)$ for $x,y\in\m H.$ This completes the proof.
\end{proof}

We shall use the following lemma whose proof is obvious.

\begin{lem}\label{double}
If $\{\zeta(k,l)\}$ is a bounded double sequence of real numbers, then there exists a convergent subsequence $\zeta(k_r,l_m)$ such that both the iterated limits
\Bea
\lim_{r\to\infty}\big(\lim_{m\to\infty}\zeta(k_r,l_m)\big)\mbox{~and~} \lim_{m\to\infty}\big(\lim_{r\to\infty}\zeta(k_r,l_m)\big)
\Eea
exist and both are equal to the double limit $\displaystyle\lim_{r,m\to\infty}\zeta(k_r,l_m).$
\end{lem}
\begin{lem}\label{rev}
If $T, U$ are bounded operators on a Hilbert space $\m H$ such that $U$ is a unitary and the sequence $\{T{U^*}^k\}$ converges to $0$ in the strong operator topology, then $T=0.$
\end{lem}
\begin{proof}
We prove this by showing that $TT^*h=0$ for every $h\in\m H.$ Let $\{P_l\}$ be a sequence of finite rank projections which converges in the strong operator topology to ${\rm Id}_{\m H}$ as $l \rightarrow \infty$. Then   $\|AP_lBh-ABh\|\leq\|A\|\|P_lBh-Bh\|\to 0$ as $l\to\infty,$ for $A, B\in\m B(\m H)$ and $h\in\m H.$ Hence, we conclude that
\bea\label{lim1}
\lim_{l\to\infty}\|T{U^*}^kP_lU^kT^*h\|=\|TT^*h\| \mbox{~for every fixed~} k\in\mb N \mbox{~and~} h\in\m H.
\eea
Since $P_l$ is a finite rank operator for each $l\in \mb N,$ applying Lemma \ref{conj} with $A_k=B_k=T{U^*}^k,$ we obtain from the hypothesis that
\Bea
\lim_{k\to\infty}\|T{U^*}^kP_lU^kT^*\|=0 \mbox{~for every fixed~} l\in\mb N.
\Eea
In particular, we have
\bea\label{lim2}
\lim_{k\to\infty}\|T{U^*}^kP_lU^kT^*h\|=0 \mbox{~for every fixed~} l\in\mb N \mbox{~and~} h\in\m H.
\eea
For fixed $h\in\m H,$ define the double sequence $\zeta:\mb N\times\mb N\to\mb R$ by $\zeta(k,l)=\|T{U^*}^kP_lU^kT^*h\|.$
From Equation \eqref{lim1} and Equation \eqref{lim2}, we have
\bea\label{sub}
\lim_{k\to\infty}\big(\lim_{l\to\infty}\zeta(k,l)\big)=\|TT^*h\| \mbox{~and~} \lim_{l\to\infty}\big(\lim_{k\to\infty}\zeta(k,l)\big)=0,
\eea
respectively.
Since   $\{\zeta(k,l)\}$ is a bounded double sequence of real numbers, by Lemma \ref{double}, there exists a convergent subsequence $\zeta(k_r,l_m)$ such that both the iterated limits
\Bea
\lim_{r\to\infty}\big(\lim_{m\to\infty}\zeta(k_r,l_m)\big)\mbox{~and~} \lim_{m\to\infty}\big(\lim_{r\to\infty}\zeta(k_r,l_m)\big)
\Eea
exist and both are equal to the double limit $\displaystyle\lim_{r,m\to\infty}\zeta(k_r,l_m).$ Therefore, by Equation \eqref{sub}
\Bea
\|TT^*h\|=\lim_{r\to\infty}\big(\lim_{m\to\infty}\zeta(k_r,l_m)\big)=\lim_{m\to\infty}\big(\lim_{r\to\infty}\zeta(k_r,l_m)\big)=0.
\Eea
Hence $\|TT^*h\|=0,$ so, $TT^*h=0.$   This completes the proof.
\end{proof}

 \begin{lem}\label{ayv}
  Let $U$ and $V$ be a unitary and a pure isometry on Hilbert spaces $\mathcal H_1, \mathcal H_2$ respectively, and let $T:\mathcal H_1\to\mathcal H_2$ be a bounded operator such that $ V^*TU=T.$ Then $T = 0$.
 \end{lem}
\begin{proof}
By iteration, we get from hypothesis that ${V^*}^nTU^n=T$ for every positive integer $n$.  Therefore, $ T{U^*}^n={V^*}^nT.$ Since $V$ is a pure isometry, the sequence $\{{V^*}^n\}$ converges to $0,$ in the strong operator topology. Therefore, the sequence $\{T{U^*}^n\}$ converges to $0,$ in the strong operator topology. So, the proof follows from Lemma \ref{rev}.

\end{proof}

\textbf{We are now ready to prove Theorem \ref{Wold}.}

By the Wold decomposition of an isometry, we may write $S_n = U\oplus V$ on $\mathcal H = \mathcal H_1\oplus\mathcal H_2,$ where $\mathcal H_1,\mathcal H_2$ are reducing subspaces for $S_n$, the operator $S_n|_{\mathcal H_1} = U$ is unitary and the operator $S_n|_{\mathcal H_2} = V$ is a pure isometry.
Let us write
 $$
 S_i = \begin{bmatrix} S_{11}^{(i)} & S_{12}^{(i)}\\ S_{21}^{(i)} & S_{22}^{(i)} \end{bmatrix},\, i = 1,\ldots,n-1.
 $$
 with respect to this decomposition, where $S_{jk}^{(i)}$ is a bounded operator from $\mathcal H_k$ to $\mathcal H_j.$ Now,
 \Bea
S_n^*S_iS_n &=&  \begin{bmatrix} U^* & 0\\ 0 & V^* \end{bmatrix}\begin{bmatrix} S_{11}^{(i)} & S_{12}^{(i)}\\ S_{21}^{(i)} & S_{22}^{(i)} \end{bmatrix}  \begin{bmatrix} U & 0\\ 0 & V \end{bmatrix}\\
&=& \begin{bmatrix}U^* S_{11}^{(i)}U & U^* S_{12}^{(i)}V\\V^* S_{21}^{(i)}U & V^* S_{22}^{(i)}V \end{bmatrix} ,\, i = 1,\ldots,n-1.
 \Eea
Note that
$$S_n^*S_i=S_{n-i}^*\iff S_n^*S_{n-i}=S_i^* \mbox{(replacing $i$ by $n-i$)} \iff S_i=S^*_{n-i}S_n.$$
Putting $S_i=S^*_{n-i}S_n$ in $S_n^*S_i=S_{n-i}^*$, we obtain $S_n^*S_{n-i}^*S_n=S_{n-i}^*$ which is the same as $S_n^*S_{n-i}S_n=S_{n-i}$. In other words,
\begin{equation} \label{ar2}
S_n^*S_i S_n=S_i \mbox{ for all } i=1,2, \ldots ,n-1.
\end{equation}  Using this, we have
\begin{enumerate}
\item[(i)] $U^* S_{12}^{(i)}V=S_{12}^{(i)},$
\item[(ii)] $V^* S_{21}^{(i)}U= S_{21}^{(i)}.$
\end{enumerate}
Clearly, (i) is equivalent to $V^* {S_{12}^{(i)}}^*U={ S_{12}^{(i)}}^*,$  hence by Lemma \ref{ayv}, $ {S_{12}^{(i)}}^*=0,$ so, $S_{12}^{(i)}=0.$ Another application of Lemma \ref{ayv} together with (ii), shows that  $ { S_{21}^{(i)}}^*=0.$ So,
\Bea
S_i= \begin{bmatrix} S_{11}^{(i)} & 0 \\ 0 & S_{22}^{(i)} \end{bmatrix},\, i = 1,\ldots,n-1.
\Eea
Since $S_n^*S_i=S_{n-i}^*,$ we have $U^*S_{11}^{(i)}={S^{(n-i)}_{11}}^*$ and $V^*S_{22}^{(i)}={S_{22}^{(n-i)}}^*$.

  The relation \eqref{ar2} remains true for both the reduced tuples
  $$(S_1|_{\mathcal H_1}, \ldots ,S_{n-1}|_{\mathcal H_1}, U) \mbox{ and } (S_1|_{\mathcal H_2}, \ldots ,S_{n-1}|_{\mathcal H_2}, V).$$
  For the first one, the relation \eqref{ar2} means $U^* S_i|_{\mathcal H_1} U = S_i|_{\mathcal H_1}$ for all $i.$ Since $U$ is a unitary, commutativity follows.

  Now we prove part (c) of the theorem, i.e., the structure of the second tuple above. Since $V$ is a pure isometry, it is unitarily equivalent to the shift on $H^2(\mathcal D_{V^*})$. This unitary equivalence is implemented by a unitary $W$ mentioned in the statement of the theorem. To avoid cumbersome notation, we put $S_{22}^{(i)} = V_i$ for $i=1,2, \ldots , n-1$. The relation \eqref{ar2} gives
  $$V_i = V_{n-i}^* V = V^*V_iV \mbox{ for } i=1,2, \ldots , n-1.$$
   Since dim $\mathcal D_{V^*} =$ dim $\mathcal D_{S_n^*} < \infty$, the operators $WV_1W^*, WV_2W^*, \ldots , WV_{n-1}W^*$ are Toeplitz operators with symbols $\v_1, \v_2, \ldots \v_{n-1}$ from  $L^\infty\big(\m B (\m D_{V^*})\big)$. This is where finite dimensionality of $\mathcal D_{S_n^*} = \mathcal D_{V^*}$ is being used.

  The relations between the symbols that are satisfied because of \eqref{ar2} are $\v_{n-i}(z) = \v_i(z)^*z$. Let $\v_{i}(z)=\sum_{k=-\infty}^\infty A_k^{(i)}z^k$ be the Fourier expansion of $\v_i$ for $\{A_k^{(i)}\}_{k=-\infty}^\infty \subseteq \m B(\m D_{V^*})$ and $\vert z\vert=1.$ Then,
\Bea
\sum_{k=-\infty}^\infty A_k^{(n-i)}z^k=\sum_{k=-\infty}^\infty (A_k^{(i)})^*z^{-k+1}
=\sum_{k=-\infty}^\infty (A_{-k}^{(i)})^*z^{k+1}=\sum_{k=-\infty}^\infty (A_{-k+1}^{(i)})^*z^k,\,\, \vert z\vert=1.
\Eea
So, $A_k^{(n-i)}=(A^{(i)}_{-k+1})^*$ for all $k\in\mb Z.$ Define $f_{i}(z)=\sum_{k=1}^\infty A_k^{(i)}z^{k-1}$. Then
\begin{eqnarray*}\v_i(z) & = & zf_i(z) + \sum_{k=-\infty}^0 A_k^{(i)} z^k \\
& = & zf_i(z) + \sum_{k=1}^\infty A_{-k+1}^{(i)} z^{-k+1} = zf_i(z) + \sum_{k=1}^\infty (A_{k}^{(n-i)})^* z^{-k+1} = zf_i(z) + f_{n-i}(z)^*.\end{eqnarray*}

To compute, $ S_{n-i}^* - S_i S_n^*$, we note that
\begin{eqnarray*} S_{n-i}^* - S_i S_n^* & = & \left(
                             \begin{array}{cc}
                               (S_{n-i}|_{\mathcal H_1})^* - S_i|_{\mathcal H_1} (S_n|_{\mathcal H_1})^* & 0 \\
                               0 & (S_{n-i}|_{\mathcal H_2})^* - S_i|_{\mathcal H_2} (S_n|_{\mathcal H_2})^* \\
                             \end{array}
                           \right) \\
                           & = & \left(
                                       \begin{array}{cc}
                                         (S_n|_{\mathcal H_1})^*S_i|_{\mathcal H_1} - S_i|_{\mathcal H_1} (S_n|_{\mathcal H_1})^* & 0 \\
                                         0 &  V_{n-i}^* - V_i V^*\\
                                       \end{array}    \right).\end{eqnarray*}
On the $\mathcal H_1$ part, we get $0$ by unitarity of $S_n|_{\mathcal H_1}$. For the $\mathcal H_2$ part, using the form of the $\v_i$, we get
$$  V_{n-i}^* - V_i V^* = T_{\v_{n-i}}^* - T_{\v_i} T_z^* =    T_{f_{n-i}}^*   T_z^* + T_{f_{i}} - (T_z T_{f_{i}} + T_{f_{n-i}}^*) T_z^* = T_{f_i} (I - T_z T_z^*).$$
Hence \eqref{determinant} follows. Uniqueness of the tuple $(f_1, f_2, \ldots ,f_{n-1})$ follows from   \eqref{determinant} by virtue of the fact that any $\m B (\m D_{V^*})$ valued function $\boldmath f$ is uniquely determined by the action of $T_{\boldmath f}$ on the space $\m D_{V^*}$ which is in fact the subspace of $H^2 (\m D_{V^*})$ consisting of $\m D_{V^*}$ valued $constant$ functions.

The uniqueness of the decomposition follows from uniqueness in the Wold decomposition of the isometry $S_n$. This completes the proof of the theorem. \qed

\section{Consequences of the Wold decomposition}

As an immediate consequence of the Wold decomposition theorem, we get a structure theorem for a pure Agler-Young isometry.

\begin{cor} \label{PureStructure}
Let $\underline{S} = (S_1, S_2, \ldots S_n)$ be a pure Agler-Young isometry with $\dim \mathcal D_{S_n^*} < \infty$. Then there is a function tuple $\bl f = (f_1, f_2, \ldots f_{n-1})$ from $H^\infty(\mathcal B (\mathcal D_{S_n^*}))$ such that $\underline{S}$ is unitarily equivalent (by a unitary $W,$ say) to the canonical Agler-Young isometry associated with $\boldmath f$. Moreover,
\begin{equation} S_{n-i}^* - S_i S_n^* =  W^*\left( T_{f_i} (I - T_zT_z^*) \right)W.\end{equation} \end{cor}
\begin{proof} The proof follows from part (c) of Theorem \ref{Wold}. \end{proof}

 It is important to note the structure of a commuting Agler-Young isometry. See also Theorem 4.10 of \cite{SS}. We give a different proof here. For two bounded operators $T_1$ and $T_2$, the notation $[T_1, T_2]$ denotes the commutator $T_1T_2 - T_2T_1$.

\begin{cor}
Let $(S_1, S_2. \ldots,S_n)$ be a commuting Agler-Young isometry on $\mathcal H$.
Then there is a unique orthogonal decomposition $\mathcal H = \mathcal H_1 \oplus \mathcal H_2$ of the Hilbert space $\mathcal H$ such that
\begin{enumerate}[(a)]
\item $\mathcal H_1$ and $\mathcal H_2$ are common reducing subspaces for the $S_i$,
 \item $(S_1|_{\m H_1}, S_2|_{\m H_1}, \ldots ,S_n|_{\m H_1})$ is a commuting Agler-Young unitary.
 \item  $S_n|_{\mathcal H_2}$ is a pure isometry $V$. There is a unitary operator $W:\m H_2\to H^2(\m D_{V^*})$ and a unique $(n-1)$-tuple $(X_1, X_2, \ldots , X_{n-1})$ of operators on $\mathcal B(\mathcal D_{V^*})$ satisfying \begin{equation} \label{simple} [X_i, X_j] = 0 \mbox{ and } [X_j, X_{n-i}^*] = [X_i, X_{n-j}^*] \mbox{ for } 1 \le i,j \le n-1 \end{equation}
 such that $ WS_i|_{\m H_2}W^*$ is the multiplication on $H^2(\m D_{V^*})$ by $zX_i + X_{n-i}^*$. Further, the following relation is satisfied for every $i =1, 2, \ldots , n-1$:
\begin{equation*}  S_{n-i}^* - S_i S_n^* = 0_{\mathcal H_1} \oplus W^*\left(  (I - T_zT_z^*)^{1/2} X_i (I - T_zT_z^*)^{1/2} \right)W.\end{equation*}

\end{enumerate} \label{comAYiso}
\end{cor}
\begin{proof}
We already know the decomposition from the Wold decomposition theorem. Moreover, the finite dimensionality condition is not required because of commutativity. Recall that the finite dimensionality of $\mathcal D_{S_n^*}$ was used to infer that $S_1|_{\mathcal H_2}, S_2|_{\mathcal H_2}, \ldots ,S_{n-1}|_{\mathcal H_2}$ were Toeplitz operators. In the present context, this conclusion is immediate from commutativity. However, commutativity also brings in severe constraints. Since $ WS_i|_{\m H_2}W^*$ now commutes with a shift, it is an analytic Toeplitz operator. But its symbol is of the form $zf_i(z) + f_{n-i}(z)^*$. Hence, analyticity forces the $f_i$ to be constant, say, $X_i$. Thus, we get the symbol of $ WS_i|_{\m H_2}W^*$ to be $zX_i + X_{n-i}^*$. Now, we invoke commutativity of $ WS_i|_{\m H_2}W^*$ with $ WS_j|_{\m H_2}W^*$ for $i,j=1,2, \ldots, n-1$. So, $zX_i + X_{n-i}^*$ has to commute with $zX_j + X_{n-j}^*$. And this gives equation \eqref{simple}. The last assertion follows from \eqref{determinant} by noting that a constant multiplier leaves the range of the projection $(I - T_zT_z^*)$ invariant. \end{proof}

\begin{cor} The restriction of an Agler-Young unitary to a common invariant subspace is an Agler-Young isometry. Conversely, if $\underline{S}$ is an Agler-Young isometry with $\dim\mathcal D_{S_n^*} < \infty$, then it is the restriction of an Agler-Young unitary $\underline{R}$ to a common invariant subspace.
\end{cor}
\begin{proof}
If an Agler-Young unitary is restricted to a common invariant subspace, then the restriction is clearly an Agler-Young isometry.

Conversely, given an Agler-Young isometry $\underline{S} = (S_1, S_2, \ldots,S_n)$ on $\mathcal H$ with $\dim\mathcal D_{S_n^*} < \infty$, we have the Wold decomposition $\mathcal H = \mathcal H_1 \oplus \mathcal H_2$ as in Theorem \ref{Wold}. On the reducing subspace $\mathcal H_1$, the restriction of $\underline{S}$ is an Agler-Young unitary. We know the structure of $\underline{S}$ restricted to the reducing subspace $\mathcal H_2$ from Theorem \ref{Wold}. Without loss of generality, take $\mathcal H_2$ to be $H^2 (\m D_{V^*})$ so that we can omit the unitary $W$ in the following discussion. Identify $H^2 (\m D_{V^*})$ as a subspace of $L^2 (\m D_{V^*})$, the Hilbert space of $\m D_{V^*}$ valued square integrable functions on the unit circle, define $\mathcal K$ to be $\mathcal H_1 \oplus L^2 (\m D_{V^*})$ and
$$\underline{R} = (S_1|_{\mathcal H_1} \oplus M_{\v_1},S_{2}|_{\mathcal H_1} \oplus M_{\v_2}, \ldots , , \ldots ,S_{n-1}|_{\mathcal H_1}  \oplus M_{\v_{n-1}}, U \oplus M_z)$$
where $\v_1, \v_2, \ldots , \v_{n-1}$ are as in Theorem \ref{Wold}.
  That completes the proof.
\end{proof}

\begin{rem}

The restriction of a commuting Agler-Young unitary to a common invariant subspace is a commuting Agler-Young isometry. Conversely, if $\underline{S}$ is a commuting Agler-Young isometry, then it is the restriction of a commuting Agler-Young unitary. No finite dimensionality assumption is required. This is because of Corollary \ref{comAYiso} which provides just the right Wold decomposition that is required. Indeed, in presence of commutativity, we take $\mathcal K = \mathcal H_1 \oplus L^2 (\m D_{V^*})$ and $\underline{R} = (S_1|_{\mathcal H_1} \oplus M_{\v_1},S_{2}|_{\mathcal H_1} \oplus M_{\v_2}, \ldots , , \ldots ,S_{n-1}|_{\mathcal H_1}  \oplus M_{\v_{n-1}}, U \oplus M_z)$ where $\v_i$ now is the analytic function $zX_i + X_{n-i}^*$ obtained from Corollary \ref{comAYiso}. \end{rem}

We proceed to give a set of complete invariants for Agler-Young isometries. Two operator tuples $\underline{A} = (A_1, A_2, \ldots , A_{n-1}, A_n)$ and $\underline{B} = (B_1, B_2, \ldots , B_{n-1}, B_n)$ acting on Hilbert spaces $H$ and $K$ respectively are called unitarily equivalent if there is a single unitary operator $U : H \rightarrow K$ such that $B_i = UA_iU^*$ for each $i=1,2, \ldots ,n.$ The following definition is in the same spirit.

\begin{defn} Let $\mathcal E $ and $\mathcal E^\prime$ be Hilbert spaces. Consider two sets of fundamental functions $f_1, f_2, \ldots ,f_{n-1}$ and $g_1, g_2, \ldots ,g_{n-1}$ from $H^\infty\big(\mathcal B(\mathcal E)\big)$ and $H^\infty\big(\mathcal B(\mathcal E^\prime)\big)$ respectively. They are called unitarily equivalent if there is a single unitary operator $U : \mathcal E \rightarrow \mathcal E^\prime$ such that $g_i(z) = Uf_i(z) U^*$ for every $z \in \mathbb D$ and each $i=1,2, \ldots ,n$. \end{defn}

\begin{prop}
If two sets of fundamental functions are unitarily equivalent, then their associated pure Agler-Young isometries are unitarily equivalent. Conversely, if $\underline{A}$ and $\underline{B}$ are two pure Agler-Young isometries with $\dim \mathcal D_{A_n^*} < \infty$ and $\dim \mathcal D_{B_n^*} < \infty$ and if $\underline{A}$ and $\underline{B}$ are unitarily equivalent, then their fundamental functions are unitarily equivalent. \end{prop}
\begin{proof}
Suppose we have two Hilbert spaces $\cle$ and $\cle^\prime$ and two sets of functions: $f_1, f_2, \ldots ,f_{n-1}$ from $H^\infty(\mathcal B(\mathcal E))$ and $g_1, g_2, \ldots ,g_{n-1}$ from $H^\infty(\mathcal B(\mathcal E^\prime))$ with the assumption that there is a unitary $U:\m E\to\m E^\i$ such that $Uf_i(z)U^*=g_i(z)$ for $z\in\mb T$ and $i=1,\ldots, n-1.$. Let $\v_i(z)=zf_i(z)+f_{n-i}(z)^*$ and $\psi_i(z)=zg_i(z)+g_{n-i}(z)^*$ for $i=1,\ldots, n-1$. These co-analytic extensions are unitarily equivalent too because $$U\v_i(z)U^* = U(zf_i(z) + f_{n-i}(z)^*)U^* = zg_i(z) + g_{n-i}(z)^* = \psi_i(z).$$

Therefore, considering the Fourier expansions $\v_i(z)=\sum_{k=-\infty}^\infty \alpha_k^{(i)}z^k$ and $\psi_i(z)=\sum_{k=-\infty}^\infty \beta_k^{(i)}z^k,$ where $\{\alpha_n^{(i)}\}_{n=-\infty}^\infty\subseteq \m B(\m E)$ and $\{\beta_k^{(i)}\}_{k=-\infty}^\infty\subseteq \m B(\m E^\i),$ we have
\begin{equation}\label{coeff}
U\alpha_k^{(i)}U^*=\beta_k^{(i)}\mbox{~ for~} i=1,\ldots,n-1 \mbox{~and~} k\in\mb Z.
\end{equation}
Let $h\in L^2(\m E)$ have the Fourier expansion $h(z)=\sum_{n=-\infty}^\infty h_nz^n$ for $\{h_n\}_{n=-\infty}^\infty\subseteq\m E.$ Define a unitary $\tilde U:L^2(\m E)\to L^2(\m E^\i)$  by $(\tilde Uh)(z)=\sum_{n=-\infty}^\infty (Uh_n)z^n$.   If $P_+:L^2(\m E)\to H^2(\m E)$ and $P_+^\i:L^2(\m E^\i)\to H^2(\m E^\i)$ denote the canonical projections, then it is easy to verify that $\tilde UP_+=P^\i_+\tilde U.$
Consider $(T_{\v_1},T_{\v_2}, \ldots , T_{\v_{n-1}}, T_z)$ and $(T_{\psi_1},T_{\v_2}, \ldots , T_{\psi_{n-1}}, T_z)$ acting on $H^2(\cle)$ and $H^2(\cle^\prime)$ respectively. It is also true that $\tilde UT_{\v_i}=T_{\psi_i}\tilde U$ for $i=1,\ldots,n-1.$ Indeed, $\tilde UP_+M_{\v_i}=P_+^\i M_{\psi_i}\tilde U$ because
\begin{align}
(\tilde UP_+M_{\v_i}h)(z) &= P_+^\i\tilde U(M_{\v_i}h)(z) \nonumber \\
&=P_+^\i\sum_{n=-\infty}^\infty U\big(\sum_{k=-\infty}^\infty \alpha^{(i)}_kh_{n-k}\big)z^n \nonumber \\
&= P_+^\i\sum_{n=-\infty}^\infty \big(\sum_{k=-\infty}^\infty U\alpha^{(i)}_kU^*Uh_{n-k}\big)z^n \nonumber \\
&=P_+^\i\sum_{n=-\infty}^\infty \big(\sum_{k=-\infty}^\infty \beta^{(i)}_kUh_{n-k}\big)z^n=P_+^\i(M_{\psi_i}\tilde Uh)(z), \nonumber
\end{align}
proving that the Agler-Young isometries are unitarily equivalent.

Conversely, if two pure Agler-Young isometries $\underline{A}$ and $\underline{B}$ are unitarily equivalent with the finite dimensionality assumptions mentioned above, we know that $\underline{A}$ and $\underline{B}$ are unitarily equivalent to two canonical pure Agler-Young isometries. Let the Fourier coefficients of the corresponding $\v_i$ and $\psi_i$ be $\alpha_k^{(i)}$ and $\beta_k^{(i)}$ respectively. Then \eqref{coeff} has to hold. Then obviously, the fundamental functions are unitarily equivalent.
\end{proof}

The following corollary, which is an immediate consequence of the proposition and Corollary \ref{PureStructure}, is a far reaching generalization of \cite[Corollary 3.2]{S}. In the commuting case, the assumption about finite dimensionality of the defect spaces is not required, see Corollary 5.2 of \cite{SS}.
\begin{cor}
Two pure Agler-Young isometries $\underline{A} = (A_1, A_2, \ldots ,A_n)$ and $\underline{B} = (B_1, B_2, \ldots ,B_n)$ with $\dim \mathcal D_{A_n^*} < \infty$ and $\dim \mathcal D_{B_n^*} < \infty$ are unitarily equivalent if and only if the two $(n-1)$-tuples
$$(A_1^*-A_{n-1}A_n^*, A_2^*-A_{n-2}A_n^*, \ldots , A_{n-1}^* - A_1A_n^*)$$
and
$$(B_1^*-B_{n-1}B_n^*, B_2^*-B_{n-2}B_n^*, \ldots , B_{n-1}^* - B_1B_n^*)$$  are unitarily equivalent.
\end{cor}

We shall end this section with a neat result which characterizes pure Agler-Young isometries with a remarkable simplicity.

\begin{prop}\label{TheCs}
 Let $S_n$ be a pure isometry and let $(C_1, C_2, \ldots ,C_{n-1})$ be a tuple of bounded operators  such that each $C_i$ commutes with either $S_n$ or $S_n^*$. Let $S_i=C_iS_n+C_{n-i}^*$. Then $\underline{S} = (S_1, S_2, \ldots , S_n)$ is a pure Agler-Young isometry.

 Conversely, if $\underline{S} = (S_1, S_2, \ldots , S_n)$ is a pure Agler-Young isometry with $\dim \mathcal D_{S_n^*} < \infty$, then there exists a tuple $(C_1, C_2, \ldots ,C_{n-1})$ of bounded operators  such that each $C_i$ commutes with either $S_n$ or $S_n^*$ and $S_i=C_iS_n+C_{n-i}^*$. \end{prop}

\begin{proof}
If $S_n$ is a pure isometry and $S_i=C_iS_n+C_{n-i}^*$  such that each $C_i$ commutes with either $S_n$ or $S_n^*$, then
\begin{align} S_{n-i}^* S_n = (C_{n-i}S_n + C_i^*)^*S_n &= S_n^* C_{n-i}S_n + C_iS_n \\
&= \left\{ \begin{array}{c} C_{n-i} S_n^*S_n + C_iS_n \mbox{ if } C_{n-i} \mbox{ commutes with } S_n^* \\  S_n^*S_n C_{n-i} + C_iS_n \mbox{ if } C_{n-i} \mbox{ commutes with } S_n \end{array} \right. \end{align}
In either case, we get $S_{n-i}^* S_n = S_i$. Hence $\underline{S}$ is a pure Agler-Young isometry.

Conversely, if $(S_1, S_2, \ldots , S_n)$ is a pure Agler-Young isometry with $\dim \mathcal D_{S_n^*}<\infty$, then by Corollary \ref{PureStructure}, $S_n$ is a pure isometry and there is a function tuple $\bl f = (f_1, f_2, \ldots f_{n-1})$ from $H^\infty(\mathcal B (\mathcal D_{S_n^*}))$ and a unitary $W: \mathcal H \rightarrow H^2 (\mathcal D_{S_n^*})$ such that $(WS_1W^*, WS_2W^*, \ldots , WS_nW^*)$ is equal to the canonical Agler-Young isometry associated with $\bl f$. Let $C_i = W^*T_{f_i}W$ for $i=1,2, \ldots ,n-1$. Then the $C_i$ commute with $S_n$ and
$$S_i = W^*(WS_iW^*)W = W^*(T_{f_i}T_z + T_{f_{n-i}}^*)W = C_iS_n + C_{n-i}^*.$$
 \end{proof}

 \begin{rem}
 If $(S_1, S_2, \ldots , S_n)$ is a commuting pure Agler-Young isometry (with no finite dimensionality assumption), then also there exists a tuple $(C_1, C_2, \ldots ,C_{n-1})$ of bounded operators  such that each $C_i$ commutes with either $S_n$ or $S_n^*$ and $S_i=C_iS_n+C_{n-i}^*$. In fact, $C_i = W^*T_{f_i}W$ for $i=1,2, \ldots ,n-1$ where the functions $f_i$ are the constants $X_i$ obtained in Corollary \ref{comAYiso}. \end{rem}

We shall return to Agler-Young isometries in Section 6 when we show that their adjoints are the $extremals$ of the $family$ of adjoints of elements from the Agler-Young class.

\section{Proof of the dilation theorem (Theorem 2)}

In a nutshell, the content of this section is to prove the following.

\vspace*{5mm}

 \begin{centerline}
{\em $\underline{S} = (S_1, S_2, \ldots,S_n)$ is in the Agler-Young class} \end{centerline}
 \begin{centerline} {\em if and only if $\underline{S}$ has a dilation to an Agler-Young isometry.} \end{centerline}

\vspace*{5mm}

We start with a preliminary lemma.

\begin{lem} If $\underline{S} = (S_1, S_2, \ldots,S_n)$ is an $n$-tuple of bounded operators on a Hilbert space $\clh$ having an Agler-Young isometric dilation, then it has a minimal Agler-Young isometric dilation. \end{lem}

\begin{proof}
Let us start with an Agler-Young isometric dilation $\underline{W} = (W_1, W_2, \ldots,W_n)$ acting on $\clk \supset \clh$ of $\underline{S}$. Consider the subspace
$$\clk_{\rm{min}} = \overline{\rm{span}}\{ W_n^mh : h \in \clh \mbox{ and } m=0,1, \ldots \}.$$
It is obviously invariant under $W_n$. It is not invariant under the rest of the $W_i$, but we can consider the compressions of $W_i$ to $\clk_{\rm{min}}$ for $i=1,2, \ldots , n-1$. The tuple $$\underline{R} = (R_1, R_2, \ldots ,R_n) = (P_{\clk_{\rm{min}}}W_1|_{\clk_{\rm{min}}}, P_{\clk_{\rm{min}}}W_2|_{\clk_{\rm{min}}}, \ldots, W_n|_{\clk_{\rm{min}}})$$ is an Agler-Young isometric dilation of $\underline{S}$. Indeed, the restriction of $W_n$ to $\clk_{\rm{min}}$ is an isometry and
$$ (P_{\clk_{\rm{min}}}W_{n-i}|_{\clk_{\rm{min}}})^* W_n|_{\clk_{\rm{min}}} =  P_{\clk_{\rm{min}}}W_{n-i}^* W_n|_{\clk_{\rm{min}}} = P_{\clk_{\rm{min}}} W_i|_{\clk_{\rm{min}}}$$
showing that $\underline{R}$ is an Agler-Young isometry. Moreover, $\clh \subset \clk_{\rm{min}}$. The Agler-Young isometry $\underline{R}$ not only dilates $\underline{S}$, more is true. It in fact has $\clh$ as a co-invariant subspace. To see that, first note that
$$ P_\clh R_i W_n^mh  = P_\clh W_i W_n^mh = S_i S_n^m h = S_i P_\clh W_n^mh \mbox{ for } h \in \clh \mbox{ and } m=0,1, \ldots .$$
This proves that $ P_\clh R_i = S_i P_\clh$. Now, for $h \in \clh$ and $k \in \clk_{\rm{min}}$, we have
\begin{align*} \langle R_i^*h, k \rangle = \langle h, R_i k \rangle = \langle P_\clh h, R_i k \rangle &= \langle h, P_\clh R_i k \rangle \\
&= \langle h, S_i P_\clh k \rangle = \langle P_\clh S_i^* h, k \rangle = \langle S_i^* h, k \rangle.\end{align*}
That completes the proof of co-invariance. Thus, we have an Agler-Young isometric dilation $\underline{R}$ of $\underline{S}$ which moreover enjoys the property of minimality and co-invariance. \end{proof}

\begin{lem} Let $\underline{S} = (S_1, S_2, \ldots,S_n)$ be an $n$-tuple of bounded operators on a Hilbert space $\clh$ having an Agler-Young isometric dilation $\underline{W} = (W_1, W_2, \ldots,W_n)$ acting on $\clk \supset \clh$. Then $\underline{S}$ is in the Agler-Young class. \label{compression} \end{lem}

\begin{proof}
 By virtue of the lemma above, we shall assume that $\clk = \clk_{\rm{min}} = \overline{\rm{span}}\{ W_n^mh : h \in \clh \mbox{ and } m=0,1, \ldots \}$.  Now we have the advantage that $\clk$ is the space of minimal isometric dilation of the contraction $S_n$. We know that a contraction has only one minimal isometric dilation up to unitary invariance. Thus, there is a unitary
$$U: \clk \rightarrow \m K_0=\m H\oplus\m D_{S_n}\oplus\m D_{S_n}\oplus\ldots=\m H\oplus\ell^2(\m D_{S_n})$$
such that $UW_nU^* = V_n$ where $V_n$ is the following version of the minimal unitary dilation of $S_n$ (Sch$\ddot{\mbox{a}}$ffer's costruction):
$$ V_n(h_0,h_1,h_2,\ldots)  =  (S_nh_0,{D}_{S_n}h_0,h_1, h_2,\ldots).$$
This $U$ fixes $\clh$ as well and hence each $UW_iU^*$ leaves $\clh$ as an invariant subspace. Thus corresponding to the decomposition $\clk = \clh \oplus (\clk \ominus \clh)$, the operators $UW_iU^*$ have the block matrix representation:
$$UW_nU^*=\begin{pmatrix} S_n&0\\D&E
\end{pmatrix} \mbox{ and } UW_iU^* =  \begin{pmatrix} S_i&0\\D_i&E_i \end{pmatrix}  $$  where $D, D_i$ are from $ \mathcal H$ to $\ell^2(\mathcal{D}_{S_n})$ and $E, E_i$ are on $\ell^2(\mathcal{D}_{S_n}) $. Moreover,
$$D=\begin{pmatrix}{D}_{S_n} \\0\\0\\ \vdots \end{pmatrix} \mbox{ and } E=\begin{pmatrix} 0&0&0&\dots\\I&0&0&\dots\\0&I&0&\dots\\\dots&\dots&\dots&\dots \end{pmatrix} .$$
Since $(W_1, W_2, \ldots, W_{n-1},{V_n})$ is an Agler-Young isometry, we have for $i=1,2, \ldots, n-1$,
\begin{align}
 UW_iU^* & =  UW_{n-i}^*U^* UW_nU^* \nonumber \\
 \mbox{ or, } \begin{pmatrix} S_i&0\\D_i&E_i \end{pmatrix} & =  \begin{pmatrix} S_{n-i}^*&D_{n-i}^*\\0&E_{n-i}^* \end{pmatrix} \begin{pmatrix} S_n&0\\D&E
\end{pmatrix} \nonumber \\
 \mbox{ or, } \begin{pmatrix} S_i&0\\D_i&E_i \end{pmatrix} & =  \begin{pmatrix} S_{n-i}^*S_n + D_{n-i}^* D & D_{n-i}^* E  \nonumber \\ E_{n-i}^*D&E_{n-i}^*E\end{pmatrix} \nonumber
\end{align}
Out of four equations that we can get from above, we need three - the ones corresponding to $(1,1)$, $(1,2)$ and $(2,1)$ entries.
\begin{equation} \label{New11}
S_{n-i}^*S_n + D_{n-i}^* D = S_i \mbox{ for } i=1,2, \ldots, n-1.
\end{equation}
\begin{equation} \label{New12}
D_{n-i}^*E = 0.
\end{equation}
\begin{equation} \label{New21}
E_{n-i}^*D = D_i \mbox{ for } i=1,2, \ldots, n-1.
\end{equation}

From equation \eqref{New12}, we get $E^*D_{n-i} = 0$ which, because of what $E$ is, implies that only the first component of $D_{n-i}$ is non-zero. This non-zero component is an operator from $\clh$ to $\mathcal D_{S_n}$, say $Z_{n-i}$. Equation \eqref{New21} tells us that $Z_i = X_{n-i}^* D_{S_n}$ where $X_{n-i}$ is the $(1,1)$ entry of $E_{n-i}$ when written in its block matrix form as an operator on $\mathcal D_{S_n} \oplus \mathcal D_{S_n} \oplus \cdots $. Now the proof is complete in view of the equation \eqref{New11}.

 \end{proof}

 We shall now prove the converse, viz., every Agler-Young contraction has an Agler-Young isometric dilation.

\textbf{Proof of The Dilation Theorem}

(1) It is evident from the definition that $V_n$ on $\m K_0$ is the minimal isometric dilation of $S_n$ (Sch$\ddot{\mbox{a}}$ffer's construction, see \cite{Schaffer}). Let us compute the adjoints of $(V^{\bl X}_i)^*$ and $V_n^*$.  A straightforward computation shows that they are as follows.
 \Bea
&& (V^{\bl X}_i)^*(h_0,h_1,h_2,\ldots)=(S_i^*h_0+D_{S_n}X_{n-i}h_1, X_i^*h_1+X_{n-i}h_2,X_i^*h_2+X_{n-i}h_3,\ldots),\\
 && V_n^*(h_0,h_1,h_2,\ldots)=(S_n^*h_0+D_{S_n}h_1,h_2,h_3,\ldots).
 \Eea
 The Hilbert space $\m H$, embedded in $\m K_0$ by the map $h\mapsto (h,0,0,\ldots)$ is jointly co-invariant under $V^{\bl A}_i$ and $V_n$  because $(V^{\bl X}_i)^*{|_\m H}=S_i^*$ and $ V_n^*{|_\m H}=S_n^*$ for $i=1,\ldots, n-1.$

 Since $V_n$ is an isometry, in order to show that $(V_1^{\bl X},\ldots,V^{\bl X}_{n-1}, V_n)$ is an Agler-Young isometric dilation of $(S_1,\ldots,S_n)$ it is enough to verify that
 \bea
 V_n^*V_i^{\bl X}=(V_{n-i}^{\bl X})^* \mbox{~for~} i=1,\ldots, n-1.
 \eea
 For for $ i=1,\ldots, n-1,$ note that
 \Bea
&& V_n^*V_i^{\bl X}(h_0,h_1,h_2,\ldots)\\
 &=& V_n^*(S_ih_0,X_{n-i}^*{D}_{S_n}h_0+X_ih_1, X_{n-i}^*h_1+X_ih_2, X_{n-i}^*h_2+X_ih_3,\ldots )\\
 &=& (S_n^*S_ih_0+D_{S_n}X^*_{n-i}D_{S_n}h_0+D_{S_n}X_ih_1,X_{n-i}^*h_1+X_ih_2, X_{n-i}^*h_2+X_ih_3,\ldots)\\
 &=& (S^*_{n-i}h_0+D_{S_n}X_ih_1,X_{n-i}^*h_1+X_ih_2, X_{n-i}^*h_2+X_ih_3,\ldots)\\
 &=&(V_{n-i}^{\bl X})^* (h_0,h_1,h_2,\ldots),
 \Eea
 where for the penultimate equality recall that $S^*_{n-i}=S_n^*S_i+D_{S_n}X_{n-i}^*D_{S_n}$ for $i=1,\ldots, n-1.$

(2) Let us start by writing the block operator matrices of $V_1^{\bl X},V_2^{\bl X},\ldots,V_{n-1}^{\bl X},V_n$. It is evident from their defining formulae that $$V_n=\begin{pmatrix} S_n&0\\D&E
\end{pmatrix} \mbox{ and } V_i^{\bl X} =  \begin{pmatrix} S_i & 0 \\ C_i & Y_i \end{pmatrix}$$ with respect to the decomposition $\mathcal H\oplus \ell^2(\mathcal {D}_{S_n})$ of $\mathcal K_0$, where $D, C_i : \mathcal H \rightarrow \ell^2(\mathcal{D}_{S_n})$ are
$$D=\begin{pmatrix}{D}_{S_n} \\0\\0\\ \vdots \end{pmatrix} \mbox{ and } C_i = \begin{pmatrix}X_{n-i}^*{D}_{S_n} \\0\\0\\ \vdots \end{pmatrix}$$
and $E, Y_i$ on $\ell^2(\mathcal{D}_{S_n}) $ are
$$E=\begin{pmatrix} 0&0&0&\dots\\I&0&0&\dots\\0&I&0&\dots\\\dots&\dots&\dots&\dots \end{pmatrix} \mbox{ and } Y_i = \begin{pmatrix} X_i & 0 & 0 & 0 & \ldots \\
X_{n-i}^* & X_i & 0 & 0 & \ldots \\
0 & X_{n-i}^* & X_i & 0 & \ldots \\
\vdots & \vdots & \vdots & \ldots & \ldots \end{pmatrix}.$$

Let $(W_1, W_2, \ldots ,W_{n-1}, V_n)$ be any Agler-Young isometric dilation for $\underline{S} = (S_1, S_2, \ldots ,S_n)$ on $\clk_0$ such that $\clh$ is a co-invariant subspace.
Because of co-invariance of $\mathcal H$, we have the following matrix form:
$$W_i= \begin{pmatrix} S_i&0\\D_i&E_i \end{pmatrix} \mbox{ for } i=1,\ldots n-1.$$  Since $(W_1, W_2, \ldots, W_{n-1},{V_n})$ is an Agler-Young isometry, we have for $i=1,2, \ldots, n-1$,
\begin{align}
 W_i & =  W_{n-i}^* V_n \nonumber \\
 \mbox{ or, } \begin{pmatrix} S_i&0\\D_i&E_i \end{pmatrix} & =  \begin{pmatrix} S_{n-i}^*&D_{n-i}^*\\0&E_{n-i}^* \end{pmatrix} \begin{pmatrix} S_n&0\\D&E
\end{pmatrix} \nonumber \\
 \mbox{ or, } \begin{pmatrix} S_i&0\\D_i&E_i \end{pmatrix} & =  \begin{pmatrix} S_{n-i}^*S_n + D_{n-i}^* D & D_{n-i}^* E  \nonumber \\ E_{n-i}^*D&E_{n-i}^*E\end{pmatrix} \nonumber
\end{align}
 We get several equations from the above, which we list below
\begin{equation} \label{11}
S_{n-i}^*S_n + D_{n-i}^* D = S_i \mbox{ for } i=1,2, \ldots, n-1.
\end{equation}
\begin{equation} \label{12}
D_{n-i}^*E = 0.
\end{equation}
\begin{equation} \label{21}
E_{n-i}^*D = D_i \mbox{ for } i=1,2, \ldots, n-1.
\end{equation}
\begin{equation} \label{22}
E_{n-i}^* E = E_i \mbox{ for } i=1,2, \ldots, n-1.
\end{equation}
From equation \eqref{12}, we get $E^*D_{n-i} = 0$. Recalling that $E$ is really a shift, this implies that only the first component of $D_{n-i}$ is non-zero. From equation \eqref{11}, it is clear that this first component is $X_iD_{S_n}$. Hence $D_i = C_i$ so that
$$W_i = \begin{pmatrix} S_i&0\\C_i&E_i \end{pmatrix}.$$
We now have to show that $E_i = Y_i$ for $i=1,2, \ldots ,n$. Let
$$ E_i = (( \; A_{ml}^{(i)} \; ))_{m,l=1}^\infty \mbox{ for } i=1,2, \ldots, n-1.$$
The equation \eqref{22} gives us for $i=1,2, \ldots, n-1$,
\begin{align}
E^*E_{n-i} &= E_i^* \nonumber \\
\mbox{ or, } \begin{pmatrix} 0 & I & 0 & 0 & \ldots \\
0 & 0 & I & 0 & \ldots \\
0 & 0 & 0 & I & \ldots \\\
\vdots & \vdots & \vdots & \vdots & \ldots \end{pmatrix} \begin{pmatrix} A_{11}^{(n-i)} & A_{12}^{(n-i)} & A_{13}^{(n-i)} &  \ldots \\
A_{21}^{(n-i)} & A_{22}^{(n-i)} & A_{23}^{(n-i)} & \ldots \\
A_{31}^{(n-i)} & A_{32}^{(n-i)} & A_{33}^{(n-i)} &  \ldots \\
\vdots & \vdots & \vdots &  \ldots \end{pmatrix}
&= \begin{pmatrix} A_{11}^{(i)^*} & A_{21}^{(i)^*} & A_{31}^{(i)^*} &  \ldots \\
A_{12}^{(i)^*} & A_{22}^{(i)^*} & A_{32}^{(i)^*} &  \ldots \\
A_{13}^{(i)^*} & A_{23}^{(i)^*} & A_{33}^{(i)^*} &  \ldots \\
\vdots & \vdots & \vdots & \ldots \end{pmatrix} \nonumber \\
\mbox{ or, } \begin{pmatrix} A_{21}^{(n-i)} & A_{22}^{(n-i)} A_{23}^{(n-i)} &  \ldots \\
A_{31}^{(n-i)} & A_{32}^{(n-i)} A_{33}^{(n-i)} &  \ldots \\
A_{41}^{(n-i)} & A_{42}^{(n-i)} A_{43}^{(n-i)} &  \ldots \\
\vdots & \vdots & \vdots  & \ldots \end{pmatrix}
&= \begin{pmatrix} A_{11}^{(i)^*} & A_{21}^{(i)^*} & A_{31}^{(i)^*} &  \ldots \\
A_{12}^{(i)^*} & A_{22}^{(i)^*} & A_{32}^{(i)^*} &  \ldots \\
A_{13}^{(i)^*} & A_{23}^{(i)^*} & A_{33}^{(i)^*} &  \ldots \\
\vdots & \vdots & \vdots  & \ldots \end{pmatrix} \label{star}  \nonumber \\
\mbox{ or, } A_{(l+1)m}^{(n-i)} &= A_{ml}^{(i)^*} \end{align}
Hence $A_{ml}^{(i)} = A_{(l+1)m}^{(n-i)^*} = A_{(m+1(l+1)}^{(i)}$. So each $E_i$ is a Toeplitz matrix.  Let
$$ E_i = \begin{pmatrix} e_0^{(i)} & e_{-1}^{(i)} & e_{-2}^{(i)} & \ldots \\
e_1^{(i)} & e_0^{(i)} & e_{-1}^{(i)} & \ldots \\
e_2^{(i)} & e_1^{(i)} & e_0^{(i)} & \ldots \\
\vdots & \vdots & \vdots & \ldots \end{pmatrix}.$$
From equation \eqref{21}, we get, for each $i$,
$$(D_{S_n}, 0 , 0 , \ldots) \begin{pmatrix} e_0^{(i)} & e_{-1}^{(i)} & e_{-2}^{(i)} & \ldots \\
e_1^{(i)} & e_0^{(i)} & e_{-1}^{(i)} & \ldots \\
e_2^{(i)} & e_1^{(i)} & e_0^{(i)} & \ldots \\
\vdots & \vdots & \vdots & \ldots \end{pmatrix} = (D_{S_n}X_i, 0 , 0 , \ldots).$$
This leads to $D_{S_n} e_0^i = D_{S_n} X_i$ and $D_{S_n} e_{-k}^{(i)} = 0$ for $k \in \mathbb N$. These two equations mean that $e_0^i =  X_i$ and $e_{-k}^{(i)} = 0$ for $k \in \mathbb N$. So
$$ E_i = \begin{pmatrix} X_i & 0 & 0 & 0 & \ldots \\
e_1^{(i)} & X_i & 0 & 0 & \ldots \\
e_2^{(i)} & e_1^{(i)} & X_i & 0 & \ldots \\
\vdots & \vdots & \vdots & \ldots & \ldots \end{pmatrix}.$$
Also, equation \eqref{star} gives us that
$$A_{21}^{(n-i)} = A_{11}^{(i)^*} = X_i^*, A_{31}^{(n-i)} = A_{12}^{(i)^*} = 0, A_{41}^{(n-i)} = A_{13}^{(i)^*} = 0, \ldots .$$
So,
$$ E_i = \begin{pmatrix} X_i & 0 & 0 & 0 & \ldots \\
X_{n-i}^* & X_i & 0 & 0 & \ldots \\
0 & X_{n-i}^* & X_i & 0 & \ldots \\
\vdots & \vdots & \vdots & \ldots & \ldots \end{pmatrix}.$$
Thus $E_i = Y_i$ and that finishes the proof.

(3) The proof of assertion (3) simply consists of noting that $W_n$ by virtue of being a minimal isometric dilation of $S_n$ is unitarily equivalent to $V_n$. Let the unitary be $U$, i.e., $UW_nU^* = V_n$. Then $(UW_1U^*, UW_2U^*, \ldots , UW_nU^*)$ is an Agler-Young isometric dilation of $(S_1, S_2, \ldots ,S_n)$ with the last component of the dilation being $V_n$. By (2) above, this means that $UW_iU^* = V_i$ and we are done.
\qed

\section{Pure Agler-Young contractions}

In case, $S_n$ is a pure contraction, that is, $S_n^{*^m}$ converges strongly to the zero operator as $m \rightarrow \infty$, we have a simpler form of the dilation. Such an Agler-Young contraction, that is, whose last component is a pure contraction, is called a pure Agler-Young contraction. The following lemma is a dilation result as well as a functional model. Note the specific structure of the Agler-Young isometry that serves as the dilation tuple. We need some background material for it.

 Let $\Theta_{A}$ be the celebrated Sz.-Nagy Foias characteristic function of a contraction $A$. It is a $\mathcal B( \mathcal D_{A}, \mathcal D_{A^*})$ valued function on $\mathbb D$ defined as
$$ \Theta_{A} = [-A + z D_{A^*} (I - zA^*)^{-1} D_{A}]|_{\mathcal D_{A}}.$$
For a complete discussion of its properties and usefulness, see \cite{SNF}. The function $\Theta_A$ induces a multiplier $M_{\Theta_A}$ from $H^2 (\mathcal D_{A})$ into $H^2 ( \mathcal D_{A^*})$, i.e.,
$$ (M_{\Theta_A} f)(z) = \Theta_A(z) f(z) \mbox{ for } f \in H^2 (\mathcal D_{A}).$$
If $A$ is pure, then $M_{\Theta_A}$ is an isometry.

Sz.-Nagy and Foias showed that every pure contraction, say $A$, defined on a Hilbert space $\mathcal H$ is unitarily equivalent to the operator
$$\mathbb {A} =P_{\mathbb H_{A}}(T_z)|_{\mathbb H_{A}} \mbox{ on the Hilbert space } \mathbb H_{A}=(H^2( \mathcal D_{A^*}) \ominus M_{\Theta_A}(H^2(\mathcal D_A)).$$ This is known as the Sz.Nagy-Foias model for a pure contraction. We can use their result to produce the required model for a pure Agler-Young contraction. Let $\theta$ be the characteristic function of the pure contraction $S_n$, i.e., $\theta = \Theta_{S_n}$ in the notation of the above. Let us remember that $M_\theta$ is an isometry because $S_n$ is pure.

\begin{lem} \label{W}
  Let $\underline{S} = (S_1, S_2, \ldots ,S_n)$ be a pure Agler-Young contraction on $\mathcal H$. Suppose $S_n$ is not an isometry and $\dim \mathcal D_{S_n^*} < \infty$. Then there are $n-1$ bounded operators $Y_1, Y_2, \ldots , Y_{n-1}$ on $\mathcal D_{S_n^*}$ such that  $\underline{S}$ is unitarily equivalent to the commuting tuple $\mathbb{S} = (\mathbb S_1, \mathbb S_2, \ldots \mathbb S_n)$ on the function space $ \mathbb H_{\underline{S}} = H^2(\mathcal D_{S_n^*}) \ominus M_{\theta_{S_n}} H^2(\mathcal D_{S_n})$ defined by $\mathbb S_i = P_{\mathbb H_{\underline{S}}} T_{Y_i + zY_{n-i}^*}|_{\mathbb H_{\underline{S}}}$ for $1 \le i \le n-1$ and $\mathbb S_n = P_{\mathbb H_{\underline{S}}} T_z|_{\mathbb H_{\underline{S}}}$.

\end{lem}

\begin{proof}
For any $Y_1, Y_2, \ldots , Y_{n-1} \in \mathcal B(\mathcal D_{S_n^*})$, the tuple
 $$(T_{Y_1 + zY_{n-1}^*}, T_{Y_2 + zY_{n-2}^*}, \ldots , T_{Y_{n - 1}+ zY_1^*}, T_z)$$
 is a canonical Agler-young isometry. Indeed, the associated $f_1, f_2, \ldots ,f_{n-1}$ are constant functions $f_i(z) = Y_i, i=1,2, \ldots ,n-1$.

We shall show that $\underline{S}$ dilates to such an Agler-Young isometry by embedding $\mathcal H$ isometrically into $H^2(\mathcal D_{S_n^*})$ via an isometry $W$ as a proper co-invariant subspace for $\underline{T}$ and showing that
\begin{equation} \label{identification} WS_i^*W^* = T_{Y_i + zY_{n-i}}^*|_{W\mathcal H}, i=1,2, \ldots n-1 \mbox{ and } WS_n^*W^* = T_z^* |_{W\mathcal H}.\end{equation}
Under the isometry, the space $\mathcal H$ is identified with the range of $W$ in $H^2(\mathcal D_{S_n^*})$ and an operator $A$ on $\mathcal H$ is identified with $WAW^*$ on the range of $W$. Hence equation \eqref{identification} will mean that $S_i^*$ is unitarily equivalent to  $T_{Y_i + zY_{n-i}}^*|_{W\mathcal H}$ for $i=1,2, \ldots n-1$ and $S_n^*$ is unitarily equivalent to $T_z^* |_{W\mathcal H}$. That will prove the statement of the lemma.

The isometry $W$  is defined as $(Wh) (z) =  D_{S_n^*} (I - zS_n^*)^{-1} h$.  If we expand the right hand side of the definition of $Wh$, we get the function $\sum_{k=0}^\infty (D_{S_n^*} (S_n^*)^k h) z^k$. Its norm in  $H^2 ( \mathcal D_{S_n^*})$ is
$$\sum_{k=0}^\infty \| D_{S_n^*} (S_n^*)^k h \|^2 = \sum_{k=0}^\infty \langle S_n^k D_{S_n^*}^2 (S_n^*)^k h , h\rangle .$$
This is a telescopic sum and equals $ \| h \|^2 - \lim_{k \rightarrow \infty} \| (S_n^*)^k h  \|^2 = \| h \|^2$. Thus $W$ is an isometry.

If $f(z) = \sum_{k=0}^\infty a_k z^k$ with $a_k \in \mathcal D_{S_n^*}$ is an arbitrary element of $H^2 ( \mathcal D_{S_n^*})$, then for $h \in \mathcal H$,
\begin{align*} \langle W^*f, h \rangle &= \langle W^*(\sum_{k=0}^\infty a_k z^k) , h \rangle \\
&= \langle \sum_{k=0}^\infty a_k z^k , Wh \rangle \\
&= \langle \sum_{k=0}^\infty a_k z^k , \sum_{k=0}^\infty (D_{S_n^*} (S_n^*)^k h) z^k \rangle \\
&= \sum_{k=0}^\infty \langle a_k , D_{S_n^*} (S_n^*)^k h \rangle = \sum_{k=0}^\infty \langle S_n^k D_{S_n^*} a_k ,   h \rangle \end{align*}
so that $W^*f = \sum_{k=0}^\infty  S_n^k D_{S_n^*} a_k$. It immediately follows from this computation that $W^*T_z = S_nW^*$ because $(T_zf)(z) = \sum_{k=0}^\infty a_k z^{k+1} = \sum_{k=1}^\infty a_{k-1} z^{k}$ so that
$$ W^*T_z f = \sum_{k=1}^\infty  S_n^{k} D_{S_n^*} a_{k-1} = S_n \sum_{k=1}^\infty  S_n^{k-1} D_{S_n^*} a_{k-1} = S_nW^*f.$$
 Hence $W\mathcal H$ is a co-invariant subspace of $T_z$. Moreover, $WS_n^*W^* = T_{z}^*|_{W\mathcal H}$. Thus $T_z$ is the minimal isometric dilation of $WS_nW^*$. Consequently, by uniqueness of minimal isometric dilation of a contraction, there is a unitary $U: \mathcal K_0 \rightarrow H^2 ( \mathcal D_{S_n^*})$ such that $UV_nU^* = T_z$ where $\mathcal K_0$ and $V_n$ are as in the last section. This $U$ also fixes $\mathcal H$, i.e., the image under $U$ of the subspace $\mathcal H \oplus 0 \oplus 0 \oplus \ldots $ of $\mathcal K_0$ is $W\mathcal H$. Now, $(UV_1U^*, \ldots ,UV_{n-1}U^*, T_z)$ is an Agler-Young isometry that leaves $W\mathcal H$ co-invariant. Since the last component of this Agler-Young isometry is $T_z$, we know from Corollary \ref{PureStructure} that it is a canonical Agler-Young isometry $\underline{T} = (T_{\v_1}, \ldots ,T_{\v_{n-1}}, T_z)$. It is an Agler-Young isometric dilation of the given $\underline{S}$ because $(V_1, V_2, \ldots ,V_n)$ is so. Range of $W$ is a proper subspace because otherwise $S_n$ will be a shift of some multiplicity, but by assumption it is not an isometry.

To reach the special structure of the $\v_i$ as mentioned in the statement, we need to note that there is a relation between $W$ and $M_\theta$, viz.,
$$WW^* + M_\theta M_\theta^* = I.$$ We are not proving this here in detail. The proof can be done by applying $WW^* + M_\theta M_\theta^*$ on vectors of $H^2(\mathcal D_{S_n^*})$ of the form $\zeta/(1 - z\overline{w})$ (where $\zeta \in \mathcal D_{S_n^*}$) and can also be easily found in the literature, see Lemma 3.3 of \cite{BhP} for example.
Since $W$ and $M_\theta$ both are isometries (in presence of pureness of $S_n$), $WW^*$ and $M_\theta M_\theta^*$ are complementary orthogonal projections. Now it follows from the computations we have done above involving $W$ that there is a unitary between $\mathcal H$ and the range of $W$ which is $\mathbb H_{S_n}$. Moreover, this unitary which is just $W$ mapping $\mathcal H$ onto $\mathbb H_{S_n}$ also conjugates the operators rightly:
$$WS_i^*W^* = T_{\v_i}^*|_{\mathbb H_{S_n}} \mbox{ and } WS_n^*W^* = T_z^*|_{\mathbb H_{S_n}}.$$

 The range of $M_\theta$, which is automatically closed being the range of an isometry, and which equals $(W\mathcal H)^\perp = {\mathbb H_{S_n}}^\perp$ is an invariant subspace of the canonical Agler-Young isometry $(T_{\v_1}, \ldots ,T_{\v_{n-1}}, T_z)$. Thus we are lead to the situation that we have a non-trivial invariant subspace of $T_z$ which is also invariant under the Toeplitz operators $T_{\v_1}, T_{\v_2}, \ldots , T_{\v_{n-1}}$. The non-triviality of the invariant subspace is due to the fact that $S_n$ is not an isometry, i.e., not a shift. A Toeplitz operator and the operator $T_z$ can have a common non-trivial invariant subspace only if the Toeplitz operator has an analytic symbol, see \cite{Nakazi}. Since $\underline{T}$ is a canonical Agler-Young isometry, the $\v_i$ are co-analytic extensions of some $f_1, f_2, \ldots ,f_{n-1}$ which are in fact holomorphic. The only way $\v_i$ can be analytic is if the $f_i$ are constants, say, $Y_i$. Thus, $\v_i(z) = Y_i + zY_{n-i}^*$ for $i=1,2, \ldots ,n-1$.  \end{proof}

The theorem above is a far reaching non-commutative generalization of Theorem 3.1 of \cite{BhP}. In general, an Agler-Young tuple does not enjoy the property that its adjoint tuple is an Agler-Young tuple, any canonical Agler-Young isometry with non-constant $\bl{f}$ is such an example. However, the theorem above allows us to conclude the following about the adjoint tuple.

 \begin{cor}
 Let $\underline{S} = (S_1, S_2, \ldots ,S_n)$ be a pure Agler-Young contraction on $\mathcal H$. Suppose $S_n$ is not an isometry and $\dim \mathcal D_{S_n^*} < \infty$. Then the adjoint tuple $\underline{S}^* = (S_1^*, S_2^*, \ldots ,S_n^*)$ is an Agler-Young contraction.
  \end{cor}

 \begin{proof}
 From the dilation result above for a pure Agler-Young contraction, we infer that there are $Y_1, Y_2, \ldots , Y_{n-1} \in \mathcal B ( \mathcal D_{S_n^*})$ such that equation \eqref{identification} holds where $W : \mathcal H \rightarrow H^2 ( \mathcal D_{S_n^*})$ is as defined before. These $Y_i$ will be shown to satisfy
 $$S_i^* - S_{n-i}S_n^* = D_{S_n^*} Y_i^* D_{S_n^*}.$$

 For simpler computations, we shall use the identification of the Hilbert space $H^2 ( \mathcal D_{S_n^*})$ with the tensor product $H^2 \otimes  \mathcal D_{S_n^*}$ where $H^2$ is the Hardy space of scalar valued functions on $\mathbb D$. In this picture, $T_z$ denotes multiplication by $z$ on $H^2$ and
 $$ Wh = \sum_{k=0}^{\infty} z^k \otimes D_{S_n^*} S_n^{*K}h, W^*(T_z\otimes I) = S_nW^* \mbox{ and } W^*(I \otimes Y_i + T_z\otimes Y_{n-i}^*) = S_iW^*.$$
 Hence $S_i^* - S_{n-i}S_n^*$ on $\mathcal H$ is identified on the range of $W$ with
 \begin{align*}
 & W (S_i^* - S_{n-i}S_n^*)W^* |_{\rm{Ran} W} \\
 = &  (I \otimes Y_i + T_z\otimes Y_{n-i}^*)^*WW^* - P_{\rm{Ran} W} (I \otimes Y_{n-i} + T_z\otimes Y_{i}^*) (T_z\otimes I)^* WW^*|_{\rm{Ran} W} \\
 = & P_{\rm{Ran} W} (I \otimes Y_i^* + T_z^* \otimes Y_{n-i} - T_z^* \otimes Y_{n-i} - T_zT_z^*\otimes Y_{i}^*) (T_z\otimes I)^* WW^* |_{\rm{Ran} W}\\
 = & P_{\rm{Ran} W} (P_{\mathbb C} \otimes Y_i^*)|_{\rm{Ran} W} \end{align*}
 where $P_{\mathbb C}$ is the projection in $H^2$ onto the one-dimensional subspace of constants.
 The penultimate line above is reached by co-invariance of the range of $W$ by all the Toeplitz operators $I \otimes Y_i + T_z\otimes Y_{n-i}^*$ and $T_z \otimes I$. Now, it is known that $P_{\rm{Ran} W} (P_{\mathbb C} \otimes Y_i^*)|_{\rm{Ran} W} = WD_{S_n^*} Y_i^* D_{S_n^*} W^*|_{\rm{Ran} W}$, see Theorem 5.1 in \cite{S}. Hence we are done.
 \end{proof}

Since the discussion in this section so far has greatly depended on closed subspaces of $H^2(\mathcal E)$ that are invariant under $T_z$ as well as under all $T_{\v_i}$, we can ask the question of whether there is a description of such an invariant subspace. The following proposition answers that question. Recall that by Beurling-Lax-Halmos theorem, for any closed subspace $\mathcal M$ of $H^2 \otimes \mathcal E$ that is invariant under $T_z$, there is an auxiliary space $\mathcal F$ and a $\mathcal B ( \mathcal F , \mathcal E)$ valued inner function $\theta$ on $\mathbb D$ (in fact, the $\theta$ could be taken to be the characteristic function of a certain pure contraction) such that $\mathcal M =$ Ran$ M_\theta$. This $\theta$ is called the Beurling-Lax-Halmos function of the subspace $\mathcal M$.

\begin{prop}
  Let $\mathcal E$ be a Hilbert space, let $f_1, f_2, \ldots ,f_{n-1}$ be functions from $H^\infty(\mathcal B (\mathcal E))$ and let $\v_1, \v_2, \ldots ,\v_{n-1}$ be their co-analytic extensions. Let $\mathcal M$ be a closed subspace of $H^2 \otimes \mathcal E$ that is invariant under $T_z$. Let $\theta$ be the Beurling-Lax-Halmos function of $\mathcal M$ with $\mathcal F$ being the auxiliary space. Then $\mathcal M$ is invariant under all the $T_{\v_i}$ if and only if there is a unique tuple $(g_1, g_2, \ldots , g_{n-1})$ from $H^\infty(\mathcal B (\mathcal F))$ such that its co-analytic extension tuple $(\psi_1, \psi_2, \ldots ,\psi_{n-1})$ satisfies
  $$ \v_i(z) \theta(z) = \theta(z) \psi_i(z) \mbox{ for all } z \in \mathbb D  \mbox{ and for all } i=1,2, \ldots ,n-1.$$
\end{prop}

\begin{proof}
Clearly, the condition is sufficient for $\mathcal M$ to be simultaneously invariant under $T_{\v_1}, T_{\v_2}$, \ldots ,$T_{\v_n}$. It is the necessity that we need to prove. To that end, we do the following computation involving operators on $H^2(\mathcal F)$.
$$(M_\theta^* T_{\v_{n-i}} M_\theta)^* T_z = M_\theta^* T_{\v_{n-i}}^* M_\theta T_z = M_\theta^* T_{\v_{n-i}}^* T_z M_\theta = M_\theta^* T_{\v_i} M_\theta.$$
Consequently, the tuple $(M_\theta^* T_{\v_1} M_\theta, M_\theta^* T_{\v_2} M_\theta, \ldots ,M_\theta^* T_{\v_n} M_\theta)$
is a pure Agler-Young isometry. By Corollay \ref{PureStructure}, it is a canonical Agler-Young isometry ($T_{\psi_1}, T_{\psi_2}$, \ldots ,$T_{\psi_n}$), say. Since $M_\theta^* T_{\v_i} M_\theta = T_{\psi_i}$, we have $M_\theta M_\theta^* T_{\v_i} M_\theta = M_\theta T_{\psi_i}$. By virtue of the fact that $\mathcal M$ is an invariant subspace for $T_{\v_i}$, the projection $M_\theta M_\theta^*$ in the last equation is redundant. Hence $T_{\v_i} M_\theta = M_\theta T_{\psi_i}$. \end{proof}

\section{A von Neumann type inequality}

von Neumann proved that for any contraction $T$ and any polynomial $p$, one has
\begin{equation} \label{vN} \| p(T) \| \le \| p \|_\infty \end{equation}
where $\| p \|_\infty = \sup \{ |p(z)| : z \in \mathbb D\}$.  This is a characterization of contractions that led to the study of spectral and complete spectral sets. The class that we are studying, viz., the Agler-Young class, is defined by a system of operator equations. Does a von Neumann type inequality as in Equation \eqref{vN} characterize the Agler-Young class? This is the question we shall answer in this section by falling back on an argument which originated in \cite{BhPSR} as a beautiful application of the operator version of Fejer-Riesz Theorem. In the following, $w(n)$ denotes a constant, depending on $n$.

\begin{lem} \label{hered} The following are equivalent.
\begin{enumerate}
\item $\underline{S} = (S_1, S_2, \ldots ,S_n) \in AY_n$ with the numerical radius of each $X_i$ being not greater than $w(n)$, a constant depending on $n$,
\item $ w(n)(I - S_n^* S_n) \ge \mathrm{Re} (\exp{i\theta} (S_i - S_{n-i}^*S_n) \mbox{ for all } \theta \in [0,2\pi).$ \end{enumerate} \end{lem}

\begin{proof} The proof uses the following lemma.

\begin{lem}[Lemma 4.1 of \cite{BhPSR}] Let $\Sigma$ and $D$ be two bounded operators on $\clh$. Then
$$DD^* \ge \mathrm{Re} (\exp{i\theta} \Sigma) \mbox{ for all } \theta \in [0,2\pi)$$
if and only if there is an $F \in \mathcal B(\mathcal D_*)$ with numerical radius of $F$ not greater than one such that $\Sigma = DFD^*$, where $\mathcal D_* = \overline{\rm Ran} D_*$. \end{lem}

 For our purpose, let $\Sigma_i = S_i - S_{n-i}^*S_n$. Assuming (1) above, we know that $\Sigma_i = D_{S_n} X_i D_{S_n}$ for some $X_i$ with $w(X_i) \le w(n)$. Hence by the lemma above, we have
$$ w(n)(I - S_n^* S_n) \ge Re (\exp{i\theta} (S_i - S_{n-i}^*S_n) \mbox{ for all } \theta \in [0,2\pi).$$
Conversely, if we assume (2) above and want to prove (1), we apply the lemma again which guarantees the existence of an $X_i \in \mathcal B (\mathcal D_{S_n})$ such that $\Sigma_i = D_{S_n} X_i D_{S_n}$ and $w(X_i) \le w(n)$.
\end{proof}

The characterization obtained in Lemma \ref{hered} allows us to link the Agler-Young class to one of Agler's landmark paper \cite{AglerFamily} where he outlined an abstract approach to model theory.

\begin{defn}
Let $\mathcal P$ be the ring of all polynomials over the complex field in the non-commuting variables $(\underline{z}, \underline{z}^* )= (z_1, z_2, \ldots ,z_n, z_1^*, z_2^*, \ldots ,z_n^*)$. The involution on the algebra $\mathcal P$ is:
$$(z_i)^* = z_i^*, (z_i^*)^* = z_i \mbox{ and } (uv)^* = v^* u^*$$
for $i=1,2, \ldots ,n$ and for any words $u, v$ in the non-commuting variables. A polynomial is called hereditary if in its monomials, all $z_i^*$ appear before all the $z_j$. \end{defn}

The hereditary polynomials have found many uses in operator theory ever since they were introduced by Agler in \cite{AglerFamily}, we mention here a relevant few.

\begin{enumerate}

\item A contraction $T$ is characterized by $h(T, T^*) \ge 0$ where $h(z, z^*) = 1 - z^* z$,
\item A spherical contraction $\underline{T} = (T_1, T_2, \ldots T_n)$ is characterized by $h(\underline{T}, \underline{T}^*) \ge 0$ where $h(\underline{z}, \underline{z}^*) = 1 - z_1^* z_1 - \cdots - z_n^*z_n$ (see \cite{RS}),
\item A $\Gamma_n$-contraction, that is, a commuting tuple of bounded operators $\underline{T} = (T_1, T_2, \ldots T_n)$ having the symmetrized polydisc as a spectral set satisfies $h(\underline{T}, \underline{T}^*) \ge 0$ where $h(\underline{z}, \underline{z}^*) = \sum_{i,j=0}^n (-1)^{i+j} \{n- (i+j)\} z_i^*z_j$ \cite[Proposition 2.18]{SS}. \end{enumerate}

Now, we can re-write Lemma \ref{hered} as follows.

\begin{theorem}[\textbf{Characterization in terms of hereditary polynomials}] \label{ChHered}
For every $n \ge 2$, there is a set of hereditary polynomials $h_{\alpha, i}$ indexed by $ (\alpha, i) \in \mathbb T \times \{1,2, \ldots ,n\}$ such that $\underline{S} = (S_1, S_2, \ldots ,S_n) \in AY_n$ with $w(X_i) \le w(n)$ for each $i$ if and only if $h_{\alpha, i} (\underline{S}) \ge 0$. \end{theorem}

\begin{proof} Take  $h_{\alpha, i}(\underline{z}, \underline{z}^*) = 2w(n)(1 - z_n^*z_n) - \alpha (z_i - z_{n-i}^*z_n) - \overline{\alpha} (z_i^* - z_n^*z_{n-i})$. \end{proof}

It is known from Agler's work that a class characterized by hereditary polynomials must be a family. We can prove it directly for the Agler-Young class.
\begin{defn}

For $n \ge 1$, a family $\mathcal F$ is a collection of $n$-tuples $\underline{T} = (T_1, T_2, \ldots , T_n)$ of Hilbert space operators (acting on $\clh$ say), which is

\begin{enumerate}

\item[(1)] bounded, that is, $\| T_i \| \le c$ for some constant $c$ for all $i=1,2, \ldots ,n$,

\item[(2)] closed under restriction to invariant subspaces, that is, if $\underline{T} \in \mathcal F$ and if $\clm \subset \clh$ is an invariant subspace for each $T_i$, then $(T_1|_\clm, T_2|_\clm, \ldots ,T_n|_\clm) \in \mathcal F$,

\item[(3)] closed under direct sum, that is, if $\underline{T}^{(m)} \in \mathcal F$, then the tuple
$$ \underline{T} \bydef \oplus_{m=1}^\infty \underline{T}^{(m)} = (\oplus_{m=1}^\infty T^{(m)}_1, \oplus_{m=1}^\infty T^{(m)}_2, \ldots , \oplus_{m=1}^\infty T^{(m)}_n)$$
is in $\mathcal F$.

\item[(4)] closed under $*$-representation, that is, if $\pi$ is a unital $*$-representation and $\underline{T} \in \mathcal F$, then $\pi(\underline{T}) \in \mathcal F$.
\end{enumerate} \end{defn}

Agler defined it only for a single variable, although he mentioned that the concept generalizes effortlessly to several variables. Since then it has found widespread use: we mention the works of Dritschel and McCullough who used the family of $\rho$-contractions in \cite{DM} and Richter and Sundberg who used the family of commuting spherical isometries in \cite{RS}. Our concern is with the following collection:
\begin{eqnarray*}
\mathcal F_n &=& \{ (S_1^*, S_2^*, \ldots ,S_n^*) : \underline{S} = (S_1, S_2, \ldots ,S_n) \in AY_n \mbox{ with each } \\
& & S_i \mbox{ being norm bounded by a constant } c(n) \} \\
 &=& \{ (S_1^*, S_2^*, \ldots ,S_n^*) : \underline{S} = (S_1, S_2, \ldots ,S_n) \in AY_n, \| S_i \| \le c(n) \}. \end{eqnarray*}

It is straightforward that the bound $c(n)$ on the norm of each $S_i$ actually places a restriction on the numerical radius of of each $X_i$, i.e., there is a constant $w(n)$ such that $w(X_i) \le w(n)$ for each $i$.

\begin{lem} $\mathcal F_n $ is a family. \end{lem}

\begin{proof} Condition (1) is satisfied because of the constant $c$.

To see that condition (2) is satisfied, let $(S_1^*, S_2^*, \ldots ,S_n^*) \in \mathcal F_n$ and $\clm \subset \clh$ is an invariant subspace for each $S_i^*$, let $\underline{R} = (R_1, R_2, \ldots , R_n)$ be defined by $R_i^* = S_i^*|_\clm$. Since $\underline{S}  \in AY_n$, it has an Agler-Young isometric dilation $\underline{W} = (W_1, W_2, \ldots ,W_n)$ on $\clk$, say, by The Dilation theorem. From the way it was constructed, we know that $\clh$ is a co-invariant subspace for each $W_i$. Thus, the situation is that $\clk \supset \clh \supset \clm$ and $R_i^* = S_i^*|_\clm = (W_i^*|_\clh)|_\clm = W_i^*|_\clm$. Thus, $\underline{R}$ is the compression of the Agler-Young isometry to a co-invariant subspace. By Theorem \ref{compression}, $\underline{R} \in AY_n$.

For (3), a computation involving direct sums is needed. Since it is very straightforward, we omit the proof.

To prove (4), we note that unital $*$-homomorphisms preserve positivity and hence if we start with an $(S_1^*, S_2^*, \ldots , S_n^*)$ in $\mathcal F_n$, we have
$$ w(n) (I - \pi(S_n)^* \pi(S_n)) \ge \mathrm{Re} (\exp{i\theta} (\pi(S_i) - \pi(S_{n-i})^*\pi(S_n)) \mbox{ for all } \theta \in [0,2\pi)$$
which is a characterization.

\end{proof}

\begin{defn} A subset $\mathcal B$ of a family $\mathcal F$ is called a model for the family if
\begin{enumerate}
\item $\mathcal B$ is closed with respect to unital representations and direct sums,
\item for every operator tuple $\underline{T}$ in $\mathcal F$ acting on $\clh$, there is an operator tuple $\underline{B} \in \mathcal B$ acting on a bigger space $\clk \supset \clh$ such that $\clh$ is invariant under $\underline{B}$ and $T_i = B_i|_\clh$ for every $i$. \end{enumerate}
If a model $\mathcal B$ has the property that it is smallest, that is, $\mathcal B \subset \mathcal B^\prime$ for any model $\mathcal B^\prime$, then $\mathcal B$ is called a boundary.

\end{defn}
Agler showed in Theorem 5.3 of \cite{AglerFamily} that every family has a unique boundary. Consequently, it is natural to ask what the boundary is of the family consisting of conjugates of Agler-Young class operators.

\begin{lem}
For the family $\mathcal F_n$ defined above, the boundary $\mathcal B$ is given by
\begin{align} \mathcal B = &\{ (S_1^*, S_2^*, \ldots ,S_n^*) : \underline{S} = (S_1, S_2, \ldots ,S_n) \mbox{ is an Agler-Young isometry } \nonumber \\
&   \mbox{ with each } S_i \mbox{ being norm bounded by the constant } c(n) \} \nonumber \end{align} \end{lem}

\begin{proof} By the Dilation Theorem, $\mathcal B$ above is a model for $\mathcal F_n$. That it is actually the boundary will require a little more argument.

Agler defines an element $\underline{T}$ of a family to be $extremal$ if it can be the restriction of a member of family to an invariant subspace only when the invariant subspace is actually a reducing subspace and then shows that extremals of a family are always contained in a model. Indeed, if $\underline{T}$ is an extremal and $\mathcal C$ is a model, then there is an $\underline{R}$ in the model $\mathcal C$ and invariant subspace $\mathcal N$ for $\underline{R}$ such that $\underline{T} = \underline{R} |_{\mathcal N}$. But, since $\underline{T}$ is extremal, $\mathcal N$ has to be reducing. Since $\mathcal C$ is a model, by the first criterion in the definition of a model, $\underline{T}$ is in the model.

In our case, $\mathcal B$ above consists of extremals. To see it, let $\underline{T} \in \mathcal B$ and let $\underline{A}$ be an extension of it, i.e., $\underline{A}$ acts on $\clh$ and there is a subspace $\mathcal N$ such that $\underline{T} = \underline{A} |_{\mathcal N}$. Since the last component of $\underline{T}$ is a co-isometry, clearly $\mathcal{N}$ is a reducing subspace (this is the reason co-isometries form the extremals in the family of contractions). Thus with respect to the decomposition $\clh = \mathcal N \oplus \mathcal N^\perp$, we have
$$ A_i = \left(
           \begin{array}{cc}
             T_i & * \\
             0 & R_i \\
           \end{array}
         \right) \mbox{ for } i=1,2, \ldots , n-1 \mbox{ and } A_n = \left(
           \begin{array}{cc}
             T_n & 0 \\
             0 & R_n \\
           \end{array}
         \right).$$
Now, a computation of $A_i^* - A_{n-i}A_n^*$ shows that its range cannot be contained in the range of $I - A_nA_n^*$ (which it has to be because $(A_1^*, A_2^*, \ldots ,A_n^*)$ is in the Agler-Young class) unless the $(1,2)$ entries of all $A_i$ are $0$. Thus $\mathcal B$ consists of extremals and hence is contained in every model.  \end{proof}

\section{The Agler-Young class and the truncated Toeplitz operators}

We follow the notations of Sarason \cite{Sarason}. For an inner function $u$, denote by $K^2_u$ the orthocomplement $H^2 \ominus uH^2$ of the shift invariant subspace $uH^2$. In accordance with Sarason, $P$ will denote the projection from $L^2$ to $H^2$. We shall greatly use the fact that $ \mathcal{K}_u^2$ is invariant under the backward shift operator. The forward shift operator will be denoted by $T_z$ and the backward shift operator by $T_z^*$.

	Let $P_u$ denote the projection from $H^2$ onto $ \mathcal{K}_u^2$.  The truncated Toeplitz operator with symbol $\varphi$ on $ \mathcal{K}_u^2$ is defined as
$$A_\varphi = P_uT_\varphi \mid_{ \mathcal{K}_u^2} = P_uM_\varphi \mid_{ \mathcal{K}_u^2}$$
This first appeared in Sarason, \cite[p.~492]{Sarason}. Let $P_c$ denote the projection operator from $H^2$ to the one-dimensional space of constant functions.

	We shall heavily use a tool called conjugation denoted by $C$, which acts on $ \mathcal{K}_u^2$ as
$$C(g)(z) = u(z)\overline{zg(z)} , \;\; g \in   \mathcal{K}_u^2.$$
It is linear with respect to addition, $C(f+g)=C(f)+C(g)$, conjugate linear with respect to scalar multiplication, $C(af)= \bar{a}C(f)$ where $a \in \mathbb{C}$ and satisfies the following properties:
\begin{enumerate}
\item $CC(f)=f$ (involution);
\item $ \langle Cf,Cg \rangle = \langle g,f \rangle$ (antiunitary),
\item for truncated Toeplitz operators, $CA_\varphi C=A_\varphi^*$. \end{enumerate}
 Sometimes, $Cf$ will be denoted by $\tilde{f}$. Further details about $A_\varphi$ and $C$ can be found in Sarason's paper \cite[p.~495]{Sarason}.

	If $k_w$ denotes the reproducing kernel $k_w(z) = (1 - z\wbar)^{-1}$ on the Hardy space, then its projection $P_uk_w$ on $\mathcal{K}_u^2$ is denoted by $k_w^u$. Both  $k_0^u$ and $\tilde{k_0^u}$ will play significant roles for us.

	Before moving forward, it will be worthwhile to mention that whenever $\mathcal{K}_u^2$ is non trivial (i.e., is a proper non zero subspace of $H^2$), then $k_0^u \neq 0$ and consequently $\tilde{k_0^u} \neq 0$. Indeed, as $k_0^u(z) = 1- \overline{u(0)}u(z)$, (see \cite[p.~494]{Sarason}), if $k_0^u$ is constant, then $u$ has to be constant, which can never give rise to a non trivial $\mathcal{K}_u^2$. So, throughout this paper, we will assume that $\mathcal{K}_u^2$ is non trivial and thus $k_0^u \neq 0$.

\begin{theorem} Let $\varphi$ be an $L^\infty$ function. Then $(A_\varphi, A_z)$ is in the Agler-Young class if and only if the function $\varphi$ is of a particular form, viz., $\varphi= \bar{c} + cz + g$ where $g \in uH^2 + \overline{uH^2}$. In this case, $A_\varphi$ and $A_z$ commute.  \label{AY+tT}
\end{theorem}

\begin{proof}  This proof depends heavily on results from  Sarason's paper \cite{Sarason}.

 Recall from Theorem 3.1 in \cite{Sarason} that $A_\varphi = A_\psi$ if and only if $\varphi - \psi$ is in $uH^2 + \overline{uH^2}$. If $\varphi= \bar{c} + cz + g$ where $g \in uH^2 + \overline{uH^2}$, then $A_\varphi = A_{\bar c + cz}$. Thus, $A_\varphi$ is the compression of the analytic Toeplitz operator $T_{\bar c + cz}$ to the co-invariant subspace $\mathcal{K}_u^2$. Thus, $A_\varphi^* = T_{\bar c + cz}^*|_{\mathcal{K}_u^2}$. So,
$$A_\varphi^* - A_\varphi A_z^* = P_u ( T_{\bar c + cz}^* - T_{\bar c + cz} T_z^*) = c (I - A_z A_z^*) = D_{A_z^*}^{1/2} c D_{A_z^*}^{1/2}.$$
So, $(A_\varphi^*, A_z^*)$ is in the Agler-Young class. Also, $D_{A_z^*} = {k_0^u} \otimes {k_0^u}$ from Lemma 2.4 of \cite{Sarason}. Thus, $A_\varphi^* - A_\varphi A_z^* = c k_0^u \otimes k_0^u$. Now, a computation shows that
$$A_\varphi - A_\varphi^*A_z = C(A_\varphi^* - A_\varphi A_z^*)C = \bar c \tilde{k_0^u} \otimes \tilde{k_0^u}  = D_{A_z}^{1/2} \bar c D_{A_z}^{1/2}.$$
Hence, $(A_\varphi, A_z)$ is in the Agler-Young class.

The converse implication is more subtle. Let $(A_\varphi, A_z)$ be in the Agler-Young class. Write $\varphi$ as $\bar g + zh$ where $g$ and $h$ are from $H^2$. Decompose $g$ as the direct sum of two functions, one coming from $uH^2$ and the other from $\mathcal{K}_u^2$. Do the same for $h$. Then, invoke Theorem 3.1 in \cite{Sarason} to conclude that $g$ and $h$ can be taken to be in $\mathcal{K}_u^2$ without loss of generality. Let $f = g - h$. Then $f$ belongs to $\mathcal{K}_u^2$. Now, putting $\varphi_1 = \bar h + zh$, we get $\varphi = \bar g + zh = (\bar h + zh) + \bar f$, so that
\begin{align*}A_\varphi^* - A_\varphi A_z^* &= (A_{\varphi_1}^* - A_{\varphi_1} A_z^*) + (A_{\bar f}^* - A_{\bar f} A_z^*) \\
&= (A_{\varphi_1}^* - A_{\varphi_1} A_z^*) + (A_f - A_{\bar f}A_{\bar z}) = (A_{\varphi_1}^* - A_{\varphi_1} A_z^*) + A_{f - \overline{zf}}.\end{align*}
This sum is easier to analyze because of the special form of $\varphi_1$.
\begin{align*}
A_{\varphi_1}^* - A_{\varphi_1} A_z^* &= A_{\bar h + zh}^* - A_{\bar h + zh} A_z^* \\
&= A_h + A_{z\bar h} - A_{z\bar h} - A_hA_zA_z^* \\
&= A_h(I - A_zA_z^*) = A_h(k_0^u \otimes k_0^u) = P_uhk_0^u\otimes k_0^u = h\otimes k_0^u\end{align*}
because $k_0^u = 1 - \bar{u(0)}u$ and $h$ belongs to $\mathcal{K}_u^2$. Thus, $A_{\varphi_1}^* - A_{\varphi_1} A_z^*$ is the sum of a rank one operator and a truncated Toeplitz operator. On the other hand, it is also a rank one operator as the following argument shows.

The pair $(A_\varphi, A_z)$ satisfies the fundamental equation. Moreover, $D_{A_z}$ is the rank one projection $\tilde{k_0^u} \otimes \tilde{k_0^u}$ by lemma 2.4 in \cite{Sarason}. Hence,
$$A_\varphi - A_\varphi^*A_z = \bar{d_1} \tilde{k_0^u} \otimes \tilde{k_0^u}$$
for some scalar $d_1$. Conjugating both sides with $C$, we get
$$A_\varphi^* - A_\varphi A_z^* = d_1 ({k_0^u} \otimes {k_0^u}).$$

As the arguments above show, the truncated Toeplitz operator $A_{f - \overline{zf}}$ is of rank one. But the only rank one operator that it could be is a multiple of $\widetilde{k_0^u} \otimes k_0^u$, see Theorem 5.1 of \cite{Sarason}.  Hence,
$$h \otimes k_0^u + d(\widetilde{k_0^u} \otimes k_0^u) = d_1(k_0^u \otimes k_0^u).$$
In other words,
\begin{equation} \label{h} h = d_1 k_0^u - d\widetilde{k_0^u}. \end{equation}

We now need to rewrite the symbol of the rank one truncated Toeplitz operator $\widetilde{k_0^u} \otimes k_0^u$, viz., $u\bar z$ (see page 502 in \cite{Sarason}) in a certain way. We write the symbol as
$$u \bar z = (u - u(o))\bar z + u(0) \bar z = \widetilde{k_0^u} + u(0) \bar z = \widetilde{k_0^u} + u(0)\overline{z(k_0^u + \overline{u(0)}u)} = \widetilde{k_0^u} + u(0) \overline{zk_0^u} + u(0)^2 \overline{zu}.$$
Of the three terms, the last one does not contribute to the truncated Toeplitz operator because it belongs to $\overline{uH^2}$. Thus,
$$A_{f - \overline{zf}} = d(\widetilde{k_0^u} \otimes K_0^u) = A_{d\widetilde{k_0^u} + u(0)d\overline{zk_0^u}}$$
implying that the truncated Toeplitz operator for the symbol $(f - d\widetilde{k_0^u}) - \overline{z(f + \overline{u(0)d}k_0^u)}$ is $0.$ Sarason characterized the symbols that produce the zero truncated Toeplitz operator. According to the Corollary on page 499 of \cite{Sarason}, we get
\begin{equation} \label{twoequations} f - d\widetilde{k_0^u} = a k_0^u \text{ and } P_u (z(f + \overline{u(0)d}k_0^u)) = \bar a k_0^u \end{equation}
for some scalar $a$. From the first of the equations above, we get that
\begin{equation} \label{f} f = ak_0^u + d \widetilde{k_0^u}.\end{equation}
Summing equations \eqref{h} and \eqref{f}, we get that
\begin{equation} \label{g} g = f + h = ak_0^u + d_1 k_0^u. \end{equation}
Since $\varphi = \bar g + zh$, we get from equations \eqref{g} and \eqref{h} that
\begin{equation}\label{phiform} \varphi = \overline{ak_0^u + d_1 k_0^u} + z(d_1k_0^u - d\widetilde{k_0^u}) = (\overline{d_1 k_0^u} + z(d_1k_0^u)) + (\overline{ak_0^u} - dz\widetilde{k_0^u}).\end{equation}
Let $\varphi_2 = \overline{ak_0^u} - dz\widetilde{k_0^u}$. We shall now show that this symbol gives the zero truncated Toeplitz operator. To that end, we need to analyze the second equation in \eqref{twoequations}. Decompose the vector $f + \overline{u(0)} d k_0^u$ in $\mathcal{K}_u^2$ as the direct sum of a vector in the span of $\widetilde{k_0^u}$ and one that is orthogonal to $\widetilde{k_0^u}$, viz., $f + \overline{u(0)} d k_0^u = c \widetilde{k_0^u} + v_\perp$ where $c$ is a scalar and $v_\perp$ is orthogonal to $\widetilde{k_0^u}$. Then, $zv_\perp$ is in $\mathcal{K}_u^2$ (page 512 of \cite{Sarason}) and
\begin{equation} \label{P-uzetc} P_u(z \widetilde{k_0^u}) = P_u(u - u(0)) = -u(0)P_u(1) = -u(0) k_0^u. \end{equation}
Thus we get,
$$ \bar a k_0^u = P_u (z(f + \overline{u(0)d}k_0^u)) = P_u(z(c\widetilde{k_0^u} + v_\perp)) = P_u(z(c\widetilde{k_0^u}) + P_u(zv_\perp)) = -u(0) k_0^u + zv_\perp $$
implying that $zv_\perp = (\bar a + c u(0)) k_0^u$. But, by definition of $v_\perp$, we have $zv_\perp$ to be orthogonal to $k_0^u$. Hence,
\begin{equation} \label{aandc} \bar a + cu(0) = 0. \end{equation}
This also shows that $f + \overline{u(0)} d k_0^u = c\widetilde{k_0^u}$.  Recalling the value of $f$ from \eqref{f}, we get that
\begin{align}
ak_0^u + d \widetilde{k_0^u} + \overline{u(0)d} k_0^u &= c\widetilde{k_0^u} \nonumber \\
\text{Or, } (-\overline{cu(0)} + \overline{du(0)})k_0^u &= (c-d) \widetilde{k_0^u}. \label{ceqorneqd} \end{align}
This brings us to two cases.

\vspace*{5mm}

\textsf{Case - I}

In this case, $c=d$. By \eqref{aandc}, we have $a = - \overline{du(0)}$. With this value of $a$ and using \eqref{P-uzetc}, we get
$$P_u(dz\widetilde{k_0^u}) = d P_u(z\widetilde{k_0^u}) = du(0) k_0^u = \bar a k_0^u.$$
Thus,
$$A_{\varphi_2} = A_{\overline{ak_0^u} - dz \widetilde{k_0^u}} = A_{\overline{ak_0^u} - P_u (dz \widetilde{k_0^u})} =
A_{\overline{ak_0^u} - \overline{a} k_0^u} = 0.$$
Hence finally
$$A_\varphi = A_{\overline{d_1k_0^u} + zd_1 k_0^u} = A_{\overline{d_1} + zd_1}$$
using the formula $k_0^u = 1 - \overline{u(0)} u$. Thus $\varphi = \overline{d_1} + zd_1 + g$ for some $g$ in $uH^2 + \overline{uH^2}$.

\vspace*{5mm}

\textsf{Case - II}

In this case, $c \neq d$. This, in view of \eqref{ceqorneqd}, means that $\widetilde{k_0^u}$ is a multiple of $k_0^u$ which can only happen if $\mathcal{K}_u^2$ is one-dimensional. In this one-dimensional case, it is easily seen that the result is true.

\end{proof}

\section{Commutativity}

The agenda for this section is to point out that many well-studied commuting operator tuples arising from complex geometry are in Agler-Young class.

\subsection{$\Gamma_n$-contractions and Tetrablock contractions}

The {\em{tetrablock}} is defined by
$$
E = \{\underline{x}=(x_1,x_2,x_3)\in \mathbb{C}^3:
 1-x_1z-x_2w+x_3zw \neq 0 \text{ whenever }|z| < 1\text{ and }|w| < 1 \}.
$$
This is also a polynomially convex domain. A commuting triple of operators $(S_1,S_2,S_3)$ on a Hilbert space $\mathcal{H}$ is called a tetrablock contraction if $\bar{E}$ is a spectral set. These were introduced in \cite{BhTetra}.

\begin{lem} A $\Gamma_n$-contraction is in $AY_n$ and a tetrablock contraction is in $AY_3$. \end{lem}

\begin{proof} That a tetrablock contraction is in $AY_3$ was proved in \cite{BhTetra} and the proof for a $\Gamma_n$-contraction is similar with minor modifications to suit the needs. It boils down to choosing a particular family holomorphic function that leads to the hereditary polynomials $h_{\alpha,i}$ mentioned before. \end{proof}

This has gained a lot of recent attention, see \cite{SS}, \cite{APal}, \cite{SPalAug2017}. In particular, the lemma above is mentioned in \cite{SPalAug2017} and is proved in \cite{APal}.

\subsection{$\Gamma_n$-isometries and Tetrablock isometries} The distinguished boundary $b\Gamma_n$ of the symmetrized polydisc is the symmetrized torus, see Theorem 2.4 of \cite{SS}. An $n$-tuple of commuting normal operators with joint spectrum contained in $b\Gamma_n$ is called a $\Gamma_n$-unitary and the restriction of such a $\Gamma_n$-unitary to an invariant subspace is called a $\Gamma_n$-isometry. Appealing to Theorem 4.12 of \cite{SS}, we get that

\begin{lem} Up to a unitary conjugation, a commuting tuple of bounded operators $\underline{S}=(S_1, S_2, \ldots ,S_n)$ is a $\Gamma_n$-isometry if and only if $\underline{S}$ is such an Agler-Young isometry that
$$(\frac{n-1}{n}S_1, \frac{n-2}{n}S_2, \ldots ,\frac{1}{n}S_{n-1})$$
is a $\Gamma_n$-contraction. \end{lem}

The description of a tetrablock isometry is simpler. An element $\underline{x} = (x_1, x_2, x_3)$ of $\mathbb C^3$ is a member of the distinguished boundary $bE$ of the tetrablock $E$ if and only if  $x_1 = \bar{x}_2 x_3, |x_3| = 1$ and $|x_2| \le 1$. A commuting triple $\underline{N}=(N_1,N_2,N_3)$ of normal operators is called a tetrablock unitary if the joint spectrum is contained in $bE$. A tetrablock isometry is the restriction of a tetrablock unitary to a common invariant subspace. Tetrablock isometries were characterized in Theorem 5.7 of \cite{BhTetra}. In the language of Agler-Young isometries, a commuting tuple of bounded operators $(S_1, S_2, S_3)$ is a tetrablock isometry if and only if it is an Agler-Young isometry and all $S_i$ are contractions.

\subsection{Commuting dilation} It is well-known that a commuting Agler-Young tuple need not have a commuting Agler-Young dilation, see \cite{SPalCounterexample}. We end by noting that a commuting Agler-Young isometric dilation exists under a condition. This $constrained$ dilation has been observed in \cite{BhTetra}, \cite{APal} and \cite{SPalAug2017}.

\begin{lem}
Let $\underline{S} = (S_1, S_2, \ldots ,S_n)$ be in the Agler-Young class. Suppose, moreover, that the $S_i$ commute with each other and the fundamental operator tuple $X_1, X_2, \ldots , X_{n-1}$ of $\underline{S}$ satisfies \eqref{simple}. Then $\underline{S}$ has a commuting Agler-Young isometric dilation.

\end{lem}

\begin{proof} Suppose the fundamental operator tuple satisfies the given conditions. Then the Agler-Young isometric dilation $\underline{V}^{\bl X} = (V_1^{\bl X}, V_2^{\bl X}, \ldots ,V_{n-1}^{\bl X}, V_n)$ constructed in The Dilation Theorem is a commuting tuple. This is a consequence of two computations - one to show that $V_i^{\bl X}$ commutes with $V_j^{\bl X}$ for $i,j =1,2, \ldots ,n-1$ and another to show that $V_i^{\bl X}$ commutes with $V_n$ for $i=1,2, \ldots ,n-1$. Although the computations are not trivial, similar computations have been done in the literature because the commutative theory has been pursued now for many years and hence we do not repeat them here, see for example the proof Theorem 6.1 in \cite{BhTetra}.

 \end{proof}

 This lemma immediately implies known dilation theorems of $\Gamma$-contractions and tetrablock contractions.

\end{document}